\documentclass[11pt]{article}
\usepackage[english]{babel}
\usepackage[latin1]{inputenc}
\usepackage[T1]{fontenc}

\usepackage{tikz}

\usepackage{verbatim}

\usetikzlibrary{calc}
\usetikzlibrary{shapes.geometric}

\tikzset{
  c/.style={every coordinate/.try}
}

\usetikzlibrary{arrows,shapes,positioning}
\usetikzlibrary{decorations.markings}
\tikzstyle arrowstyle=[scale=1]
\tikzstyle directed=[postaction={decorate,decoration={markings,mark=at position 0.6 with {\arrow[arrowstyle]{stealth};}}}]
\tikzstyle reverse directed=[postaction={decorate,decoration={markings,mark=at position 0.4 with {\arrowreversed[arrowstyle]{stealth};}}}]
\tikzstyle dot=[style={circle,inner sep=1pt,fill}]

\usepackage{graphics,graphicx} 

\usepackage{latexsym}
\usepackage{amsfonts}
\usepackage{amsthm}
\usepackage{amsmath}
\usepackage{eepic,epic}
\usepackage{graphicx}
\usepackage[hmargin=2.3cm,vmargin=2.5cm]{geometry}

\usepackage{numprint}

\DeclareMathOperator{\asc}{\mathrm{asc}}

\newcommand\p{\circle*{0.2}}

\newtheorem{theorem}{Theorem}

\newtheorem{lemma}{Lemma}

\title{Word-representability of triangulations of rectangular polyomino with a single domino tile}
\author{Marc Glen\thanks{School of Computer and Information Sciences, University of Strathclyde, Glasgow, G1 1HX, UK. \newline Email: marc.glen.2013@uni.strath.ac.uk}\ \ and Sergey Kitaev\thanks{School of Computer and Information Sciences, University of Strathclyde, Glasgow, G1 1HX, UK. \newline Email: sergey.kitaev@cis.strath.ac.uk}}

\begin{document}

\maketitle
\thispagestyle{empty}

\begin{abstract}  A graph $G=(V,E)$ is word-representable if there exists a word $w$ over the alphabet $V$ such that letters $x$ and $y$ alternate in $w$ if and only if $(x,y)$ is an edge in $E$. 

A recent elegant result of Akrobotu et al.~\cite{Akrobotu} states that a triangulation of any convex polyomino is word-representable if and only if it is 3-colourable. In this paper, we generalize a particular case of this result by showing that the result of Akrobotu et al.~\cite{Akrobotu} is true even if we allow a domino tile, instead of having just $1\times 1$ tiles on a rectangular polyomino. \\

 \end{abstract}

\section{Introduction}\label{intro}

Suppose that $w$ is a word and $x$ and $y$ are two distinct letters in $w$.
We say that $x$ and $y$ {\it alternate} in $w$ if the deletion of all other
letters from the word $w$ results in either $xyxy\cdots$ or $yxyx\cdots$.

A graph $G=(V,E)$ is word-representable if there exists a word $w$ over the alphabet $V$ such that letters $x$ and $y$ alternate in $w$ if and only if $(x,y)$ is an edge in $E$. For example, the cycle graph on 4 vertices labeled by 1, 2, 3 and 4 in clockwise direction can be represented by the word 14213243.  There is a long line of research on word-representable graphs, which is summarized in the upcoming book~\cite{KL}.


A graph is $k$-colourable if its vertices can be coloured in at most $k$ colours so that no pair of vertices having the same colour is connected by an edge. 

\begin{theorem}[\cite{HKP2015}]\label{thm:3col} 
All $3$-colourable graphs are word-representable.
\end{theorem}

We note that, for $k\geq 4$, there are examples of non-word-representable graphs that are $k$-colourable, but not 3-colourable. For example, the wheel $W_5$ on 6 vertices is such a graph.

A {\em polyomino} is a plane geometric figure formed by joining one or more equal squares edge to edge. Letting corners of squares in a polyomino be vertices, we can treat polyominoes as graphs. In particular, well-known {\em grid graphs} are obtained from polyominoes in this way. A particular class of graphs of our interest is related to {\em convex polyominoes}. A polyomino is said to be {\em column convex} if its intersection with any vertical line is convex (in other words, each column has no holes). Similarly, a polyomino is said to be {\em row convex} if its intersection with any horizontal line is convex. Finally, a polyomino is said to be convex if it is row and column convex.

We are interested in {\em triangulations} of a polyomino. Note that no triangulation is 2-colourable -- at least three colours are needed to colour properly a triangulation, while four colours are always enough to colour any triangulation, as it is a planar graph well-known to be 4-colourable. 

\begin{figure}[h]
\begin{center}
\begin{picture}(6,4.5)

\put(-5,0){


\put(-2.5,1.8){$T_1$=}


\put(0,4){\p} \put(2,4){\p} \put(4,4){\p} 
\put(0,2){\p} \put(2,2){\p} \put(4,2){\p} 
\put(0,0){\p} \put(2,0){\p} \put(4,0){\p} 

\put(0,0){\line(1,0){4}}
\put(0,2){\line(1,0){4}}
\put(0,4){\line(1,0){4}}
\put(0,0){\line(0,1){4}}
\put(2,0){\line(0,1){4}}
\put(4,0){\line(0,1){4}}

\put(0,0){\line(1,1){2.1}}
\put(0,2){\line(1,1){2.1}}
\put(2,0){\line(1,1){2.1}}
\put(2,4){\line(1,-1){2.1}}

}

\put(7,0){


\put(-2.5,1.8){$T_2$=}


\put(0,4){\p} \put(2,4){\p} \put(4,4){\p} 
\put(0,2){\p} \put(2,2){\p} \put(4,2){\p} 
\put(0,0){\p} \put(2,0){\p} \put(4,0){\p} 

\put(0,0){\line(1,0){4}}
\put(0,2){\line(1,0){4}}
\put(0,4){\line(1,0){4}}
\put(0,0){\line(0,1){4}}
\put(2,0){\line(0,1){4}}
\put(4,0){\line(0,1){4}}

\put(0,2){\line(1,-1){2.1}}
\put(0,4){\line(1,-1){4}}
\put(2,2){\line(1,1){2.1}}

}

\end{picture}
\caption{Non-3-colourable and non-word-representable graphs $T_1$ and $T_2$.} \label{non-repr-triang}
\end{center}
\end{figure}
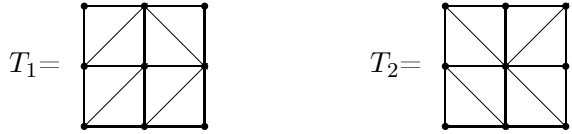

Not all triangulations of a polyomino are 3-colourable: for example, see Figure~\ref{non-repr-triang} coming from \cite{Akrobotu} for non-3-colourable triangulations, which are the only such triangulations, up to rotation, of a $3 \times 3$ grid graph. The main result in~\cite{Akrobotu} is the following theorem.

\begin{theorem}[\cite{Akrobotu}]\label{main-thm-Akrobotu} A triangulation of a convex polyomino is word-representable if and only if it is $3$-colourable. In particular, this result holds for polyominoes of rectangular shape. \end{theorem}

The basic idea of the proof of Theorem~\ref{main-thm-Akrobotu} was to show that any non-3-colourable triangulation of a convex polyomino contains a graph in Figure~\ref{non-repr-triang} as an induced subgraph. As is shown in \cite{Akrobotu}, Theorem~\ref{main-thm-Akrobotu} is not true for non-convex polyominoes. 

Inspired by the elegant result recorded in Theorem~\ref{main-thm-Akrobotu}, in this paper we consider the following variation of the problem. Polyominoes are objects formed by $1\times 1$ tiles, so that the induced graphs in question have only (chordless) cycles of length 4. A generalization of such graphs is allowing domino ($1\times 2$ or $2\times 1$) tiles to be present in polyominoes, so that in the respective induced graphs (chordless) cycles of length 6 would be allowed. We call these graphs {\em polyominoes with domino tiles}. The problem is then in characterizing those triangulations of such graphs  that are word-representable. See Figure~\ref{graphs-with-dominoes} for an example of a polyomino (of rectangular shape) with domino tiles (to the left) and one of its triangulations (to the right). 

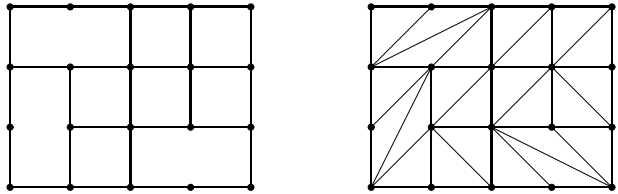
\begin{figure}[h]
\begin{center}
\begin{picture}(20,6)

\put(0,0){

\put(0,0){\p} \put(0,2){\p} \put(0,4){\p} \put(0,6){\p} 
\put(2,0){\p} \put(2,2){\p} \put(2,4){\p} \put(2,6){\p} 
\put(4,0){\p} \put(4,2){\p} \put(4,4){\p} \put(4,6){\p} 
\put(6,0){\p} \put(6,2){\p} \put(6,4){\p} \put(6,6){\p} 
\put(8,0){\p} \put(8,2){\p} \put(8,4){\p} \put(8,6){\p} 

\put(0,0){\line(1,0){8}}
\put(0,4){\line(1,0){8}}
\put(0,6){\line(1,0){8}}
\put(2,2){\line(1,0){6}}

\put(0,0){\line(0,1){6}}
\put(2,0){\line(0,1){4}}
\put(4,0){\line(0,1){6}}
\put(8,0){\line(0,1){6}}
\put(6,2){\line(0,1){4}}

}

\put(12,0){

\put(0,0){\p} \put(0,2){\p} \put(0,4){\p} \put(0,6){\p} 
\put(2,0){\p} \put(2,2){\p} \put(2,4){\p} \put(2,6){\p} 
\put(4,0){\p} \put(4,2){\p} \put(4,4){\p} \put(4,6){\p} 
\put(6,0){\p} \put(6,2){\p} \put(6,4){\p} \put(6,6){\p} 
\put(8,0){\p} \put(8,2){\p} \put(8,4){\p} \put(8,6){\p} 

\put(0,0){\line(1,0){8}}
\put(0,4){\line(1,0){8}}
\put(0,6){\line(1,0){8}}
\put(2,2){\line(1,0){6}}

\put(0,0){\line(0,1){6}}
\put(2,0){\line(0,1){4}}
\put(4,0){\line(0,1){6}}
\put(8,0){\line(0,1){6}}
\put(6,2){\line(0,1){4}}

\put(0,0){\line(1,1){2}}
\put(0,0){\line(1,2){2}}
\put(0,2){\line(1,1){2}}

\put(2,2){\line(1,1){4}}
\put(4,2){\line(1,1){4}}
\put(0,4){\line(1,1){2}}
\put(2,4){\line(1,1){2}}
\put(2,2){\line(1,-1){2}}
\put(4,2){\line(1,-1){2}}
\put(6,2){\line(1,-1){2}}
\put(4,2){\line(2,-1){4}}
\put(6,4){\line(1,-1){2}}
\put(0,4){\line(2,1){4}}

}

\end{picture}
\caption{A polyomino with domino tiles (to the left) and one of its triangulations (to the right).}\label{graphs-with-dominoes}
\end{center}
\end{figure}

In this paper, we are interested in triangulations of rectangular polyominoes with a single domino tile. Our main result is the following generalization of the case of rectangular shapes in Theorem~\ref{main-thm-Akrobotu}. 

\begin{theorem}\label{main-thm} A triangulation of a rectangular polyomino with a single domino tile is word-representable if and only if it is $3$-colourable. \end{theorem}

The first observation to make is that without loss of generality, we can assume that the single domino tile is horizontal, since otherwise, we can always rotate our rectangular polyomino 90 degrees; rotation of a shape, or taking the mirror image of it with respect to a line  are called by us {\em trivial transformations}. While the strategy below to prove Theorem~\ref{main-thm} is similar to proving Theorem~\ref{main-thm-Akrobotu}, we have to deal with many more cases resulting in 12 (non-equivalent up to trivial transformations) non-3-colourable and non-word-representable minimal graphs (which include the graphs $T_1$ and $T_2$ in Figure~\ref{non-repr-triang}) instead of just two. All these graphs, except for $T_1$ and $T_2$, are listed in Figure~\ref{non-repr-minim-triang}. 

\begin{figure}[h]
\begin{center}
\begin{picture}(27,5)

\put(0,3){

\put(-2.5,1){$A_1$=}

                    \put(1,2){\p} \put(2,2){\p} \put(3,2){\p}
\put(0,1){\p} \put(1,1){\p} \put(2,1){\p} \put(3,1){\p}
\put(0,0){\p} \put(1,0){\p} \put(2,0){\p} \put(3,0){\p}

\put(0,0){\line(1,0){3}}
\put(0,1){\line(1,0){3}}
\put(1,2){\line(1,0){2}}
\put(3,0){\line(0,1){2}}
\put(2,0){\line(0,1){2}}
\put(1,1){\line(0,1){1}}
\put(0,0){\line(0,1){1}}

\put(3,0){\line(-1,1){1}}
\put(2,1){\line(-2,-1){2}}
\put(1,1){\line(-1,-1){1}}
\put(2,1){\line(-1,-1){1}}
\put(2,1){\line(1,1){1}}
\put(2,1){\line(-1,1){1}}

}

\put(7,3){

\put(-2.5,1){$A_2$=}

                    \put(1,2){\p} \put(2,2){\p} \put(3,2){\p}
\put(0,1){\p} \put(1,1){\p} \put(2,1){\p} \put(3,1){\p}
\put(0,0){\p} \put(1,0){\p} \put(2,0){\p} \put(3,0){\p}

\put(0,0){\line(1,0){3}}
\put(0,1){\line(1,0){3}}
\put(1,2){\line(1,0){2}}
\put(3,0){\line(0,1){2}}
\put(2,0){\line(0,1){2}}
\put(1,1){\line(0,1){1}}
\put(0,0){\line(0,1){1}}

\put(3,0){\line(-1,1){1}}
\put(2,1){\line(-2,-1){2}}
\put(1,1){\line(-1,-1){1}}
\put(2,1){\line(-1,-1){1}}
\put(2,2){\line(-1,-1){1}}
\put(2,2){\line(1,-1){1}}

}

\put(14,3){

\put(-2.5,1){$A_3$=}

                    \put(1,2){\p} \put(2,2){\p} \put(3,2){\p}
\put(0,1){\p} \put(1,1){\p} \put(2,1){\p} \put(3,1){\p}
\put(0,0){\p} \put(1,0){\p} \put(2,0){\p} \put(3,0){\p}

\put(0,0){\line(1,0){3}}
\put(0,1){\line(1,0){3}}
\put(1,2){\line(1,0){2}}
\put(3,0){\line(0,1){2}}
\put(2,0){\line(0,1){2}}
\put(1,1){\line(0,1){1}}
\put(0,0){\line(0,1){1}}

\put(3,0){\line(-1,1){1}}
\put(0,1){\line(1,-1){1}}
\put(0,1){\line(2,-1){2}}
\put(1,1){\line(1,-1){1}}
\put(1,2){\line(1,-1){1}}
\put(3,2){\line(-1,-1){1}}

}

\put(21,3){

\put(-2.5,1){$A_4$=}

                    \put(1,2){\p} \put(2,2){\p} \put(3,2){\p}
\put(0,1){\p} \put(1,1){\p} \put(2,1){\p} \put(3,1){\p}
\put(0,0){\p} \put(1,0){\p} \put(2,0){\p} \put(3,0){\p}

\put(0,0){\line(1,0){3}}
\put(0,1){\line(1,0){3}}
\put(1,2){\line(1,0){2}}
\put(3,0){\line(0,1){2}}
\put(2,0){\line(0,1){2}}
\put(1,1){\line(0,1){1}}
\put(0,0){\line(0,1){1}}

\put(3,0){\line(-1,1){1}}
\put(0,1){\line(1,-1){1}}
\put(0,1){\line(2,-1){2}}
\put(1,1){\line(1,-1){1}}
\put(2,2){\line(1,-1){1}}
\put(2,2){\line(-1,-1){1}}

}

\put(28,3){

\put(-2.5,1){$A_5$=}

\put(0,2){\p} \put(1,2){\p} \put(2,2){\p} \put(3,2){\p}
\put(0,1){\p} \put(1,1){\p} \put(2,1){\p} \put(3,1){\p}
                    \put(1,0){\p} \put(2,0){\p} \put(3,0){\p}
 
\put(0,2){\line(1,0){3}}
\put(0,1){\line(1,0){3}}
\put(1,0){\line(1,0){2}}
\put(3,0){\line(0,1){2}}
\put(2,0){\line(0,1){2}}
\put(1,0){\line(0,1){1}}
\put(0,1){\line(0,1){1}}

\put(3,0){\line(-1,1){1}}
\put(0,1){\line(1,1){1}}
\put(0,1){\line(2,1){2}}
\put(1,1){\line(1,1){1}}
\put(1,1){\line(1,-1){1}}
\put(2,2){\line(1,-1){1}}

}

\put(0,0){

\put(-2.5,1){$A_6$=}

\put(0,2){\p} \put(1,2){\p} \put(2,2){\p} \put(3,2){\p}
\put(0,1){\p} \put(1,1){\p} \put(2,1){\p} \put(3,1){\p}
                    \put(1,0){\p} \put(2,0){\p} \put(3,0){\p}
 
\put(0,2){\line(1,0){3}}
\put(0,1){\line(1,0){3}}
\put(1,0){\line(1,0){2}}
\put(3,0){\line(0,1){2}}
\put(2,0){\line(0,1){2}}
\put(1,0){\line(0,1){1}}
\put(0,1){\line(0,1){1}}

\put(3,0){\line(-1,1){1}}
\put(0,2){\line(1,-1){1}}
\put(0,2){\line(2,-1){2}}
\put(1,2){\line(1,-1){1}}
\put(1,1){\line(1,-1){1}}
\put(2,2){\line(1,-1){1}}

}

\put(7,0){

\put(-2.5,1){$A_7$=}

\put(0,2){\p} \put(1,2){\p} \put(2,2){\p} \put(3,2){\p}
\put(0,1){\p} \put(1,1){\p} \put(2,1){\p} \put(3,1){\p}
\put(0,0){\p} \put(1,0){\p} \put(2,0){\p}

\put(0,0){\line(1,0){2}}
\put(0,1){\line(1,0){3}}
\put(0,2){\line(1,0){3}}
\put(3,1){\line(0,1){1}}
\put(2,0){\line(0,1){1}}
\put(1,0){\line(0,1){2}}
\put(0,0){\line(0,1){2}}

\put(0,1){\line(1,1){1}}
\put(0,1){\line(1,-1){1}}
\put(1,1){\line(1,-1){1}}
\put(1,1){\line(1,1){1}}
\put(1,1){\line(2,1){2}}
\put(2,1){\line(1,1){1}}

}

\put(14,0){

\put(-2.5,1){$A_8$=}

\put(0,2){\p} \put(1,2){\p} \put(2,2){\p} \put(3,2){\p}
\put(0,1){\p} \put(1,1){\p} \put(2,1){\p} \put(3,1){\p}
\put(0,0){\p} \put(1,0){\p} \put(2,0){\p}

\put(0,0){\line(1,0){2}}
\put(0,1){\line(1,0){3}}
\put(0,2){\line(1,0){3}}
\put(3,1){\line(0,1){1}}
\put(2,0){\line(0,1){1}}
\put(1,0){\line(0,1){2}}
\put(0,0){\line(0,1){2}}

\put(0,1){\line(1,1){1}}
\put(0,1){\line(1,-1){1}}
\put(1,1){\line(1,-1){1}}
\put(1,2){\line(1,-1){1}}
\put(1,2){\line(2,-1){2}}
\put(2,2){\line(1,-1){1}}

}

\put(21,0){

\put(-2.5,1){$B_1$=}

\put(0,2){\p} \put(1,2){\p} \put(2,2){\p} 
\put(0,1){\p} \put(1,1){\p} \put(2,1){\p} 
\put(0,0){\p} \put(1,0){\p} \put(2,0){\p} 

\put(0,0){\line(1,0){2}}
\put(0,1){\line(1,0){2}}
\put(0,2){\line(1,0){2}}
\put(0,0){\line(0,1){2}}
\put(1,1){\line(0,1){1}}
\put(2,0){\line(0,1){2}}
\put(0,1){\line(1,1){1}}
\put(1,1){\line(1,1){1}}
\put(0,1){\line(1,-1){1}}
\put(0,1){\line(2,-1){2}}
\put(1,1){\line(1,-1){1}}

}

\put(28,0){

\put(-2.5,1){$B_2$=}

\put(0,2){\p} \put(1,2){\p} \put(2,2){\p} 
\put(0,1){\p} \put(1,1){\p} \put(2,1){\p} 
\put(0,0){\p} \put(1,0){\p} \put(2,0){\p} 

\put(0,0){\line(1,0){2}}
\put(0,1){\line(1,0){2}}
\put(0,2){\line(1,0){2}}
\put(0,0){\line(0,1){2}}
\put(1,1){\line(0,1){1}}
\put(2,0){\line(0,1){2}}
\put(0,2){\line(1,-1){1}}
\put(1,2){\line(1,-1){1}}
\put(0,1){\line(1,-1){1}}
\put(0,1){\line(2,-1){2}}
\put(1,1){\line(1,-1){1}}

}

\end{picture}
\caption{All minimal (non-equivalent up to trivial transformations) non-3-colourable and non-word-representable graphs (except for $T_1$ and $T_2$ in Figure~\ref{non-repr-triang}) for triangulations of rectangular polyominoes with a single horizontal domino tile.} \label{non-repr-minim-triang}
\end{center}
\end{figure}
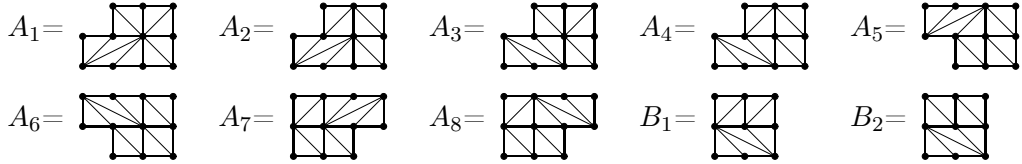

Non-word-representability of graphs in Figure~\ref{non-repr-minim-triang} follows from the following observations. First note that it follows from \cite{KP} that all odd {\em wheel graphs} $W_{2t+1}$ for $t \geq 2$ are non-word-representable, where the wheel graph $W_n$ is the graph obtained by adding an all-adjacent vertex to a cycle graph $C_n$. Using this fact, it is easy to see that all graphs in Figure~\ref{non-repr-minim-triang} are non-word-representable, as they all contain such odd wheel graphs as induced subgraphs:

\begin{itemize}
\item $A_1$ contains $W_9$ (obtained by removing the leftmost vertex on the middle horizontal line);
\item $A_2$, $A_3$, $A_6$ and $A_7$ all contain $W_7$; and
\item $A_4$, $A_5$, $A_8$, $B_1$ and $B_2$ all contain $W_5$.
\end{itemize}

The paper is organized as follows. In  Section~\ref{sec3} we will prove our main result, Theorem~\ref{main-thm}, and in Section~\ref{final-remarks-sec} we will discuss some directions of further research.

It was pointed out to us by the anonymous referee that it should be possible to reduce the number of cases in this paper to consider by using the following easy to see fact. Any triangulation involving the horizontal domino can be 3-coloured if and only if the triangulation obtained by giving the domino its other possible triangulation, leaving all other triangles unchanged, can be 3-coloured. However, we have decided to keep our original approach that has a transparent structure of all the cases that need to be considered.    


\section{Triangulations of a rectangular polyomino with a single domino tile}\label{sec3}

Let $S$ be the set of all graphs formed by rotations of the graphs in Figure~\ref{non-repr-triang}  by degrees multiple to $90\,^{\circ}$ and rotations of the graphs in Figure~\ref{non-repr-minim-triang} by degrees multiple to $180\,^{\circ}$ (only horizontal domino tiles are of interest to us).

\begin{lemma}\label{lemma-grid} A triangulation $T$ of a rectangular polyomino with a single domino tile is $3$-colourable if and only if it does not contain a graph from $S$ as an induced subgraph.  \end{lemma}

\begin{proof} If $T$ contains a graph from $S$ as an induced subgraph, then it is obviously not $3$-colourable. 

For the opposite directions, suppose that $T$ is not $3$-colourable. We note that fixing colours of the left-most top vertex in $T$ and the vertex right below it determines uniquely colours in the top two rows of $T$ (a row is a horizontal path) if we are to use colours in $\{1,2,3\}$ and keep all other vertices of $T$ uncoloured.  We continue to colour all other vertices of $T$, row by row, from left to right using any of the available colours in $\{1,2,3\}$. At some point, colour 4 must be used ($T$ is not $3$-colourable) to colour, say, vertex $v$. 

\begin{figure}[h]
\begin{center}
\begin{picture}(8,22)


\put(-12,14){

\put(-2.5,2.5){$S_1$=}

\put(2.5,1.5){{\tiny 1}} 
\put(2.5,4.1){{\tiny 2}} 
\put(5.2,1.5){{\tiny 4}} 
\put(5.2,3.5){{\tiny 3}} 

\put(3,2){\p} \put(5,2){\p} \put(3,4){\p} \put(5,4){\p} 

\put(0,0){\line(1,0){9}}
\put(0,6){\line(1,0){9}}
\put(0,2){\line(1,0){5}}
\put(3,4){\line(1,0){6}}
\put(3,2){\line(0,1){2}}
\put(5,2){\line(0,1){2}}
\put(3,4){\line(1,-1){2.1}}
\put(0,0){\line(0,1){6}}
\put(9,0){\line(0,1){6}}

\multiput(0,2)(0.5,0.5){8}{\line(1,1){0.2}}
\multiput(1,2)(0.5,0.5){8}{\line(1,1){0.2}}
\multiput(2,2)(0.5,0.5){2}{\line(1,1){0.2}}
\multiput(0,3)(0.5,0.5){6}{\line(1,1){0.2}}
\multiput(0,4)(0.5,0.5){4}{\line(1,1){0.2}}
\multiput(0,5)(0.5,0.5){2}{\line(1,1){0.2}}
\multiput(4,4)(0.5,0.5){4}{\line(1,1){0.2}}
\multiput(5,4)(0.5,0.5){4}{\line(1,1){0.2}}
\multiput(6,4)(0.5,0.5){4}{\line(1,1){0.2}}
\multiput(7,4)(0.5,0.5){4}{\line(1,1){0.2}}
\multiput(8,4)(0.5,0.5){2}{\line(1,1){0.2}}

}


\put(1,14){

\put(-2.5,2.5){$S_2$=}

\put(2.5,1.5){{\tiny 1}} 
\put(2.5,4.1){{\tiny 2}} 
\put(5.2,1.5){{\tiny 4}} 
\put(5.2,3.5){{\tiny 1}} 
\put(7.2,3.5){{\tiny 3}} 
\put(7.2,1.5){{\tiny ?}} 

\put(3,2){\p} \put(5,2){\p} \put(3,4){\p} \put(5,4){\p} \put(7,4){\p} \put(7,2){\p} 

\put(5,2){\line(1,1){2.1}}
\put(0,0){\line(1,0){9}}
\put(0,6){\line(1,0){9}}
\put(0,2){\line(1,0){7}}
\put(3,4){\line(1,0){6}}
\put(3,2){\line(0,1){2}}
\put(5,2){\line(0,1){2}}
\put(3,4){\line(1,-1){2.1}}
\put(0,0){\line(0,1){6}}
\put(9,0){\line(0,1){6}}
\put(7,2){\line(0,1){2}}

\multiput(0,2)(0.5,0.5){8}{\line(1,1){0.2}}
\multiput(1,2)(0.5,0.5){8}{\line(1,1){0.2}}
\multiput(2,2)(0.5,0.5){2}{\line(1,1){0.2}}
\multiput(0,3)(0.5,0.5){6}{\line(1,1){0.2}}
\multiput(0,4)(0.5,0.5){4}{\line(1,1){0.2}}
\multiput(0,5)(0.5,0.5){2}{\line(1,1){0.2}}
\multiput(4,4)(0.5,0.5){4}{\line(1,1){0.2}}
\multiput(5,4)(0.5,0.5){4}{\line(1,1){0.2}}
\multiput(6,4)(0.5,0.5){4}{\line(1,1){0.2}}
\multiput(7,4)(0.5,0.5){4}{\line(1,1){0.2}}
\multiput(8,4)(0.5,0.5){2}{\line(1,1){0.2}}

}


\put(14,14){

\put(-2.5,2.5){$S_3$=}

\put(2.5,1.5){{\tiny 1}} 
\put(2.5,4.1){{\tiny 3}} 
\put(5.2,1.5){{\tiny 4}} 
\put(5.2,3.5){{\tiny 2}} 
\put(7.2,3.5){{\tiny 3}} 
\put(7.2,1.5){{\tiny ?}} 

\put(3,2){\p} \put(5,2){\p} \put(3,4){\p} \put(5,4){\p} \put(7,4){\p} \put(7,2){\p} 

\put(5,2){\line(1,1){2.1}}
\put(0,0){\line(1,0){9}}
\put(0,6){\line(1,0){9}}
\put(0,2){\line(1,0){7}}
\put(3,4){\line(1,0){6}}
\put(3,2){\line(0,1){2}}
\put(5,2){\line(0,1){2}}
\put(3,2){\line(1,1){2.1}}
\put(0,0){\line(0,1){6}}
\put(9,0){\line(0,1){6}}
\put(7,2){\line(0,1){2}}

\multiput(0,2)(0.5,0.5){8}{\line(1,1){0.2}}
\multiput(1,2)(0.5,0.5){8}{\line(1,1){0.2}}
\multiput(2,2)(0.5,0.5){2}{\line(1,1){0.2}}
\multiput(0,3)(0.5,0.5){6}{\line(1,1){0.2}}
\multiput(0,4)(0.5,0.5){4}{\line(1,1){0.2}}
\multiput(0,5)(0.5,0.5){2}{\line(1,1){0.2}}
\multiput(4,4)(0.5,0.5){4}{\line(1,1){0.2}}
\multiput(5,4)(0.5,0.5){4}{\line(1,1){0.2}}
\multiput(6,4)(0.5,0.5){4}{\line(1,1){0.2}}
\multiput(7,4)(0.5,0.5){4}{\line(1,1){0.2}}
\multiput(8,4)(0.5,0.5){2}{\line(1,1){0.2}}

}


\put(-12,7){

\put(-2.5,2.5){$S_4$=}

\put(2.5,1.5){{\tiny 2}} 
\put(2.5,4.1){{\tiny 3}} 
\put(5.2,1.5){{\tiny 1}} 
\put(4.8,3.4){{\tiny 1}} 
\put(7.2,3.5){{\tiny 2}} 
\put(7.2,1.5){{\tiny 4}} 

\put(3,2){\p} \put(5,2){\p} \put(3,4){\p} \put(5,4){\p} \put(7,4){\p} \put(7,2){\p} 

\put(0,0){\line(1,0){9}}
\put(0,6){\line(1,0){9}}
\put(0,2){\line(1,0){7}}
\put(3,4){\line(1,0){6}}
\put(3,2){\line(0,1){2}}
\put(0,0){\line(0,1){6}}
\put(9,0){\line(0,1){6}}
\put(7,2){\line(0,1){2}}
\put(3,4){\line(1,-1){2}}
\put(3,4){\line(2,-1){4}}
\put(5,4){\line(1,-1){2}}

\multiput(0,2)(0.5,0.5){8}{\line(1,1){0.2}}
\multiput(1,2)(0.5,0.5){8}{\line(1,1){0.2}}
\multiput(2,2)(0.5,0.5){2}{\line(1,1){0.2}}
\multiput(0,3)(0.5,0.5){6}{\line(1,1){0.2}}
\multiput(0,4)(0.5,0.5){4}{\line(1,1){0.2}}
\multiput(0,5)(0.5,0.5){2}{\line(1,1){0.2}}
\multiput(4,4)(0.5,0.5){4}{\line(1,1){0.2}}
\multiput(5,4)(0.5,0.5){4}{\line(1,1){0.2}}
\multiput(6,4)(0.5,0.5){4}{\line(1,1){0.2}}
\multiput(7,4)(0.5,0.5){4}{\line(1,1){0.2}}
\multiput(8,4)(0.5,0.5){2}{\line(1,1){0.2}}

}

\put(1,7){

\put(-2.5,2.5){$S_5$=}

\put(2.5,1.5){{\tiny 2}} 
\put(2.5,4.1){{\tiny 1}} 
\put(5.2,1.5){{\tiny 3}} 
\put(4.8,3.4){{\tiny 2}} 
\put(7.2,3.5){{\tiny 1}} 
\put(7.2,1.5){{\tiny 4}} 

\put(3,2){\p} \put(5,2){\p} \put(3,4){\p} \put(5,4){\p} \put(7,4){\p} \put(7,2){\p} 

\put(0,0){\line(1,0){9}}
\put(0,6){\line(1,0){9}}
\put(0,2){\line(1,0){7}}
\put(3,4){\line(1,0){6}}
\put(3,2){\line(0,1){2}}
\put(0,0){\line(0,1){6}}
\put(9,0){\line(0,1){6}}
\put(7,2){\line(0,1){2}}
\put(3,4){\line(1,-1){2}}
\put(3,4){\line(2,-1){4}}
\put(5,4){\line(1,-1){2}}

\multiput(0,2)(0.5,0.5){8}{\line(1,1){0.2}}
\multiput(1,2)(0.5,0.5){8}{\line(1,1){0.2}}
\multiput(2,2)(0.5,0.5){2}{\line(1,1){0.2}}
\multiput(0,3)(0.5,0.5){6}{\line(1,1){0.2}}
\multiput(0,4)(0.5,0.5){4}{\line(1,1){0.2}}
\multiput(0,5)(0.5,0.5){2}{\line(1,1){0.2}}
\multiput(4,4)(0.5,0.5){4}{\line(1,1){0.2}}
\multiput(5,4)(0.5,0.5){4}{\line(1,1){0.2}}
\multiput(6,4)(0.5,0.5){4}{\line(1,1){0.2}}
\multiput(7,4)(0.5,0.5){4}{\line(1,1){0.2}}
\multiput(8,4)(0.5,0.5){2}{\line(1,1){0.2}}

}

\put(14,7){

\put(-2.5,2.5){$S_6$=}

\put(2.5,1.5){{\tiny 2}} 
\put(2.5,4.1){{\tiny 3}} 
\put(5.2,1.5){{\tiny 1}} 
\put(4.8,3.4){{\tiny 2}} 
\put(7.2,3.5){{\tiny 1}} 
\put(7.2,1.5){{\tiny 4}} 

\put(3,2){\p} \put(5,2){\p} \put(3,4){\p} \put(5,4){\p} \put(7,4){\p} \put(7,2){\p} 

\put(0,0){\line(1,0){9}}
\put(0,6){\line(1,0){9}}
\put(0,2){\line(1,0){7}}
\put(3,4){\line(1,0){6}}
\put(3,2){\line(0,1){2}}
\put(0,0){\line(0,1){6}}
\put(9,0){\line(0,1){6}}
\put(7,2){\line(0,1){2}}
\put(3,4){\line(1,-1){2}}
\put(3,4){\line(2,-1){4}}
\put(5,4){\line(1,-1){2}}

\multiput(0,2)(0.5,0.5){8}{\line(1,1){0.2}}
\multiput(1,2)(0.5,0.5){8}{\line(1,1){0.2}}
\multiput(2,2)(0.5,0.5){2}{\line(1,1){0.2}}
\multiput(0,3)(0.5,0.5){6}{\line(1,1){0.2}}
\multiput(0,4)(0.5,0.5){4}{\line(1,1){0.2}}
\multiput(0,5)(0.5,0.5){2}{\line(1,1){0.2}}
\multiput(4,4)(0.5,0.5){4}{\line(1,1){0.2}}
\multiput(5,4)(0.5,0.5){4}{\line(1,1){0.2}}
\multiput(6,4)(0.5,0.5){4}{\line(1,1){0.2}}
\multiput(7,4)(0.5,0.5){4}{\line(1,1){0.2}}
\multiput(8,4)(0.5,0.5){2}{\line(1,1){0.2}}

}

\put(-12,0){

\put(-2.5,2.5){$S_7$=}

\put(2.5,1.5){{\tiny 3}} 
\put(2.5,4.1){{\tiny 2}} 
\put(5.2,1.5){{\tiny 1}} 
\put(4.8,3.4){{\tiny 1}} 
\put(7.2,3.5){{\tiny 2}} 
\put(7.2,1.5){{\tiny 4}} 
\put(9.2,3.5){{\tiny 3}} 
\put(9.2,1.5){{\tiny ?}} 

\put(3,2){\p} \put(5,2){\p} \put(3,4){\p} \put(5,4){\p} \put(7,4){\p} \put(7,2){\p}  \put(9,4){\p} \put(9,2){\p}

\put(0,0){\line(1,0){11}}
\put(0,6){\line(1,0){11}}
\put(0,2){\line(1,0){9}}
\put(3,4){\line(1,0){8}}
\put(3,2){\line(0,1){2}}
\put(0,0){\line(0,1){6}}
\put(11,0){\line(0,1){6}}
\put(7,2){\line(0,1){2}}
\put(9,2){\line(0,1){2}}
\put(7,2){\line(1,1){2}}
\put(3,2){\line(1,1){2}}
\put(5,2){\line(1,1){2}}
\put(3,2){\line(2,1){4}}

\multiput(0,2)(0.5,0.5){8}{\line(1,1){0.2}}
\multiput(1,2)(0.5,0.5){8}{\line(1,1){0.2}}
\multiput(2,2)(0.5,0.5){2}{\line(1,1){0.2}}
\multiput(0,3)(0.5,0.5){6}{\line(1,1){0.2}}
\multiput(0,4)(0.5,0.5){4}{\line(1,1){0.2}}
\multiput(0,5)(0.5,0.5){2}{\line(1,1){0.2}}
\multiput(4,4)(0.5,0.5){4}{\line(1,1){0.2}}
\multiput(5,4)(0.5,0.5){4}{\line(1,1){0.2}}
\multiput(6,4)(0.5,0.5){4}{\line(1,1){0.2}}
\multiput(7,4)(0.5,0.5){4}{\line(1,1){0.2}}
\multiput(8,4)(0.5,0.5){4}{\line(1,1){0.2}}
\multiput(9,4)(0.5,0.5){4}{\line(1,1){0.2}}
\multiput(10,4)(0.5,0.5){2}{\line(1,1){0.2}}

}

\put(4,0){

\put(-2.5,2.5){$S_8$=}

\put(2.5,1.5){{\tiny 4}} 
\put(2.5,4.1){{\tiny ?}} 
\put(5.2,1.5){{\tiny ?}} 
\put(4.8,3.4){{\tiny ?}} 
\put(7.2,3.5){{\tiny ?}} 
\put(7.2,1.5){{\tiny ?}} 

\put(3,2){\p} \put(5,2){\p} \put(3,4){\p} \put(5,4){\p} \put(7,4){\p} \put(7,2){\p} 

\put(0,0){\line(1,0){9}}
\put(0,6){\line(1,0){9}}
\put(0,2){\line(1,0){7}}
\put(3,4){\line(1,0){6}}
\put(3,2){\line(0,1){2}}
\put(0,0){\line(0,1){6}}
\put(9,0){\line(0,1){6}}
\put(7,2){\line(0,1){2}}

\put(3,2){\line(1,1){2}}
\put(5,2){\line(1,1){2}}
\put(3,2){\line(2,1){4}}

\multiput(0,2)(0.5,0.5){8}{\line(1,1){0.2}}
\multiput(1,2)(0.5,0.5){8}{\line(1,1){0.2}}
\multiput(2,2)(0.5,0.5){2}{\line(1,1){0.2}}
\multiput(0,3)(0.5,0.5){6}{\line(1,1){0.2}}
\multiput(0,4)(0.5,0.5){4}{\line(1,1){0.2}}
\multiput(0,5)(0.5,0.5){2}{\line(1,1){0.2}}
\multiput(4,4)(0.5,0.5){4}{\line(1,1){0.2}}
\multiput(5,4)(0.5,0.5){4}{\line(1,1){0.2}}
\multiput(6,4)(0.5,0.5){4}{\line(1,1){0.2}}
\multiput(7,4)(0.5,0.5){4}{\line(1,1){0.2}}
\multiput(8,4)(0.5,0.5){2}{\line(1,1){0.2}}

}

\end{picture}
\caption{Eight possible cases of appearance of colour 4 in colouring of $T$.} \label{conflict-situations}
\end{center}
\end{figure}
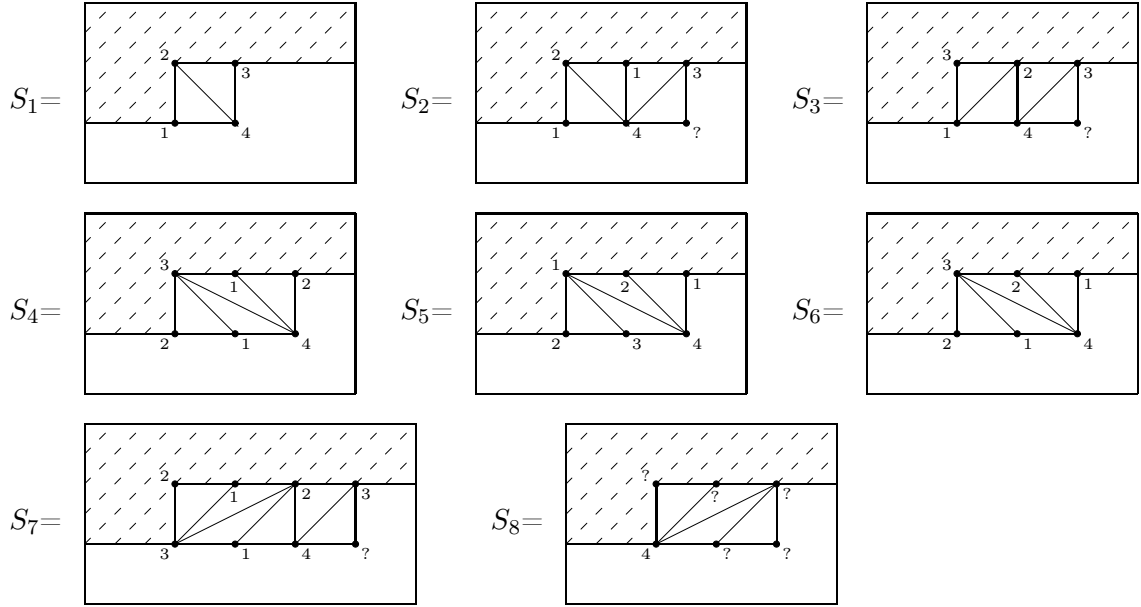

There are only eight possible different situations when this can happen, which are presented in Figure~\ref{conflict-situations} (numbers in this figure are colours). In that figure, the shaded area indicates schematically already coloured vertices of $T$, the question mark shows a still non-coloured vertex, and the colours adjacent to $v$ are fixed in a particular way without loss of generality (we can re-name already used colours if needed). A particular property in all cases is that among the colours of neighbours of $v$, we meet all the colours in $\{1,2,3\}$. Also,  by our procedure, $v$ must be in row $i$ from above, where $i\geq 3$, since the first two rows of any triangulation in question can be coloured in three colours. Note that the situations $S_1$--$S_3$ exhaust all the possibilities when the domino tile is not connected to $v$ (these cases are the same as in \cite{Akrobotu}) possibly except for $S_1$, while $S_4$--$S_8$ cover the remaining cases. 

More precisely, the situations $S_4$--$S_6$ describe all the possibilities when $v$ is the right-bottom vertex of the domino tile with the specified triangulation of it. Here $v$ is connected to four vertices, so that exactly two of them must be coloured in the same colour, giving the three cases. The situation $S_7$ is clearly the only one representing the other triangulation of the domino tile with $v$ being the right-bottom vertex. Further, $v$ cannot be the middle-bottom vertex in the domino tile, as such a vertex is connected to only two already coloured vertices, so that usage of colour 4 was not required. Finally, the Situation $S_8$ represents all cases of $v$ being the left-bottom vertex in the domino (there is no need for us to specify colours to obtain the desired result). Note that the other triangulation of the domino tile placed as in $S_8$ does not bring any new cases (because this is considered in the situation $S_1$), and thus, it is omitted.\\

\noindent 
{\bf Situation $S_1$.} We can assume that the vertices coloured 1 and 2 are not in the leftmost column, because otherwise instead of colour 1 we could use colour 3, and there would be no need to use colour 4 for colouring $v$. Further, note that in the case when  the  vertex $v$ is involved in a subgraph presented schematically to the left in Figure~\ref{S1-impossible} (the question marks there indicate that triangulations of respective squares are unknown to us), such a subgraph must be either $T_1$ or $T_2$. Indeed, otherwise, the subgraph must be one of the four graphs presented in Figure~\ref{S1-impossible}  to the right of the leftmost graph. However, in each of the four cases, we have a vertex labeled by * that would require colour 4 contradicting the fact that $v$ is supposed to be the only vertex coloured by 4. This completes our considerations of eight of subcases in the situation $S_1$ out of 36. The remaining subcases are to be considered next. 

\begin{figure}[h]
\begin{center}
\begin{picture}(35,4.5)

\put(0,0){

\put(0,4){\p} \put(2,4){\p} \put(4,4){\p} 
\put(0,2){\p} \put(2,2){\p} \put(4,2){\p} 
\put(0,0){\p} \put(2,0){\p} \put(4,0){\p}

 \put(1.4,2.3){2} \put(4.1,2.2){3} 
     \put(1.4,0.2){1} \put(4.2,0.2){4} 

\put(0,0){\line(1,0){4}}
\put(0,2){\line(1,0){4}}
\put(0,4){\line(1,0){4}}
\put(4,0){\line(0,1){4}}
\put(2,0){\line(0,1){4}}
\put(0,0){\line(0,1){4}}

\put(2,2){\line(1,-1){2}}

\put(0.5,3){{\large ?}}
\put(0.5,1){{\large ?}}
\put(3,3){{\large ?}}

}

\put(7,0){

\put(0,4){\p} \put(2,4){\p} \put(4,4){\p} 
\put(0,2){\p} \put(2,2){\p} \put(4,2){\p} 
\put(0,0){\p} \put(2,0){\p} \put(4,0){\p}

\put(-0.6,4.1){*}  \put(1.9,4.2){3} \put(4.1,4.2){1} 
\put(-0.6,2.2){1} \put(1.5,2.3){2} \put(4.1,2.2){3} 
 \put(-0.6,0.2){3}    \put(1.4,0.2){1} \put(4.2,0.2){4} 

\put(0,0){\line(1,0){4}}
\put(0,2){\line(1,0){4}}
\put(0,4){\line(1,0){4}}
\put(4,0){\line(0,1){4}}
\put(2,0){\line(0,1){4}}
\put(0,0){\line(0,1){4}}

\put(0,0){\line(1,1){4}}
\put(0,4){\line(1,-1){4}}

}

\put(14,0){

\put(0,4){\p} \put(2,4){\p} \put(4,4){\p} 
\put(0,2){\p} \put(2,2){\p} \put(4,2){\p} 
\put(0,0){\p} \put(2,0){\p} \put(4,0){\p}

\put(-0.6,4.2){?}  \put(1.9,4.2){1} \put(4.1,4.2){2} 
\put(-0.6,2.2){*} \put(1.5,2.3){2} \put(4.1,2.2){3} 
 \put(-0.6,0.2){3}    \put(1.4,0.2){1} \put(4.2,0.2){4} 

\put(0,0){\line(1,0){4}}
\put(0,2){\line(1,0){4}}
\put(0,4){\line(1,0){4}}
\put(4,0){\line(0,1){4}}
\put(2,0){\line(0,1){4}}
\put(0,0){\line(0,1){4}}

\put(0,0){\line(1,1){2}}
\put(0,2){\line(1,1){2}}
\put(2,2){\line(1,-1){2}}
\put(2,4){\line(1,-1){2}}

}

\put(21,0){

\put(0,4){\p} \put(2,4){\p} \put(4,4){\p} 
\put(0,2){\p} \put(2,2){\p} \put(4,2){\p} 
\put(0,0){\p} \put(2,0){\p} \put(4,0){\p}

\put(-0.6,4.2){2}  \put(1.9,4.2){1} \put(4.1,4.1){*} 
\put(-0.6,2.2){3} \put(1.5,2.3){2} \put(4.1,2.2){3} 
 \put(-0.6,0.2){2}    \put(1.4,0.2){1} \put(4.2,0.2){4} 

\put(0,0){\line(1,0){4}}
\put(0,2){\line(1,0){4}}
\put(0,4){\line(1,0){4}}
\put(4,0){\line(0,1){4}}
\put(2,0){\line(0,1){4}}
\put(0,0){\line(0,1){4}}

\put(0,2){\line(1,-1){2}}
\put(0,2){\line(1,1){2}}
\put(2,2){\line(1,-1){2}}
\put(2,2){\line(1,1){2}}

}

\put(28,0){

\put(0,4){\p} \put(2,4){\p} \put(4,4){\p} 
\put(0,2){\p} \put(2,2){\p} \put(4,2){\p} 
\put(0,0){\p} \put(2,0){\p} \put(4,0){\p}

\put(-0.6,4.2){1}  \put(1.9,4.1){*} \put(4.1,4.2){?} 
\put(-0.6,2.2){3} \put(1.5,2.3){2} \put(4.1,2.2){3} 
 \put(-0.6,0.2){2}    \put(1.4,0.2){1} \put(4.2,0.2){4} 

\put(0,0){\line(1,0){4}}
\put(0,2){\line(1,0){4}}
\put(0,4){\line(1,0){4}}
\put(4,0){\line(0,1){4}}
\put(2,0){\line(0,1){4}}
\put(0,0){\line(0,1){4}}

\put(0,4){\line(1,-1){4}}
\put(0,2){\line(1,-1){2}}
\put(2,4){\line(1,-1){2}}

}

\end{picture}
\caption{Impossible subcases in the situation $S_1$ in Figure~\ref{conflict-situations}.} \label{S1-impossible}
\end{center}
\end{figure}

It follows that the domino tile must share a vertex with the square coloured by 1, 2, 3 and 4 in Figure~\ref{conflict-situations} and there are 28 possible subcases to consider: 

\begin{itemize} 
\item Eight subcases, $A_1$---$A_8$, presented in Figure~\ref{non-repr-minim-triang}, where the colours of vertices are omitted (in each of these graphs, the vertex $v$ is the rightmost vertex on the bottom row);

\item 12 subcases presented in Figures~\ref{S1-impossible-01}-\ref{S1-subcases-03}. These subcases are impossible because in each of them there exists a vertex, labeled by *, that requires colour 4.

\item Four subcases presented in Figure~\ref{S1-subcases}, where colouring of vertices is shown; and, finally, 

\item Four subcases presented in Figure~\ref{S1-1-subcases}, where colouring of vertices of interest is shown. 
\end{itemize}

However, the graphs in Figure~\ref{S1-subcases}, from left to right, are, respectively:
\begin{itemize} 
\item $A_3$ flipped with respect to a horizontal line;
\item $A_1$ flipped with respect to a horizontal line;
\item $A_1$ rotated 180 degrees;
\item $A_3$ rotated 180 degrees.
\end{itemize}

\begin{figure}[h]
\begin{center}
\begin{picture}(36,4.5)

\put(0,0){

                    \put(2,4){\p} \put(4,4){\p} \put(6,4){\p}
\put(0,2){\p} \put(2,2){\p} \put(4,2){\p} \put(6,2){\p}
\put(0,0){\p} \put(2,0){\p} \put(4,0){\p} \put(6,0){\p}

                \put(1.5,4.2){2} \put(3.9,4.1){1} \put(6.2,4.1){*}
\put(-0.6,2.2){2} \put(1.4,2.3){3} \put(4.1,2.2){2} \put(6.2,2){3}
 \put(-0.6,0.2){1}    \put(1.4,0.2){3} \put(3.4,0.2){1} \put(6.2,0.2){4} 

\put(0,0){\line(1,0){6}}
\put(0,2){\line(1,0){6}}
\put(2,4){\line(1,0){4}}
\put(6,0){\line(0,1){4}}
\put(4,0){\line(0,1){4}}
\put(2,2){\line(0,1){2}}
\put(0,0){\line(0,1){2}}

\put(0,0){\line(1,1){2}}
\put(0,0){\line(2,1){4}}
\put(2,0){\line(1,1){2}}
\put(6,0){\line(-1,1){2}}

\put(2,2){\line(1,1){2}}
\put(4,2){\line(1,1){2}}

}

\put(9,0){

                    \put(2,4){\p} \put(4,4){\p} \put(6,4){\p}
\put(0,2){\p} \put(2,2){\p} \put(4,2){\p} \put(6,2){\p}
\put(0,0){\p} \put(2,0){\p} \put(4,0){\p} \put(6,0){\p}

                \put(1.5,4.2){1} \put(3.9,4.1){*} \put(6.2,4.1){?}
\put(-0.6,2.2){2} \put(1.4,2.3){3} \put(4.1,2.2){2} \put(6.2,2){3}
 \put(-0.6,0.2){1}    \put(1.4,0.2){3} \put(3.4,0.2){1} \put(6.2,0.2){4} 

\put(0,0){\line(1,0){6}}
\put(0,2){\line(1,0){6}}
\put(2,4){\line(1,0){4}}
\put(6,0){\line(0,1){4}}
\put(4,0){\line(0,1){4}}
\put(2,2){\line(0,1){2}}
\put(0,0){\line(0,1){2}}

\put(0,0){\line(1,1){2}}
\put(0,0){\line(2,1){4}}
\put(2,0){\line(1,1){2}}
\put(6,0){\line(-1,1){2}}
\put(2,4){\line(1,-1){2}}
\put(4,4){\line(1,-1){2}}

}

\put(18,0){

                    \put(2,4){\p} \put(4,4){\p} \put(6,4){\p}
\put(0,2){\p} \put(2,2){\p} \put(4,2){\p} \put(6,2){\p}
\put(0,0){\p} \put(2,0){\p} \put(4,0){\p} \put(6,0){\p}

                \put(1.5,4.1){*} \put(3.9,4.2){1} \put(6.2,4.1){2}
\put(-0.6,2.2){2} \put(1.4,2.3){3} \put(4.1,2.2){2} \put(6.2,2){3}
 \put(-0.6,0.2){1}    \put(1.4,0.2){3} \put(3.4,0.2){1} \put(6.2,0.2){4} 

\put(0,0){\line(1,0){6}}
\put(0,2){\line(1,0){6}}
\put(2,4){\line(1,0){4}}
\put(6,0){\line(0,1){4}}
\put(4,0){\line(0,1){4}}
\put(2,2){\line(0,1){2}}
\put(0,0){\line(0,1){2}}

\put(2,0){\line(-1,1){2}}
\put(4,0){\line(-2,1){4}}
\put(4,0){\line(-1,1){2}}
\put(6,0){\line(-1,1){2}}
\put(2,4){\line(1,-1){2}}
\put(4,4){\line(1,-1){2}}

}

\put(27,0){

                    \put(2,4){\p} \put(4,4){\p} \put(6,4){\p}
\put(0,2){\p} \put(2,2){\p} \put(4,2){\p} \put(6,2){\p}
\put(0,0){\p} \put(2,0){\p} \put(4,0){\p} \put(6,0){\p}

                \put(1.5,4.2){?} \put(3.9,4.2){3} \put(6.2,4.2){1}
\put(-0.6,2.2){?} \put(1.4,2.1){*} \put(4.1,2.2){2} \put(6.2,2){3}
 \put(-0.6,0.2){?}    \put(1.4,0.2){?} \put(3.4,0.2){1} \put(6.2,0.2){4} 

\put(0,0){\line(1,0){6}}
\put(0,2){\line(1,0){6}}
\put(2,4){\line(1,0){4}}
\put(6,0){\line(0,1){4}}
\put(4,0){\line(0,1){4}}
\put(2,2){\line(0,1){2}}
\put(0,0){\line(0,1){2}}

\put(2,0){\line(-1,1){2}}
\put(4,0){\line(-2,1){4}}
\put(4,0){\line(-1,1){2}}
\put(6,0){\line(-1,1){2}}

\put(2,2){\line(1,1){2}}
\put(4,2){\line(1,1){2}}

}

\end{picture}
\caption{Four impossible subcases in the situation $S_1$ in Figure~\ref{conflict-situations}.} \label{S1-impossible-01}
\end{center}
\end{figure}

\begin{figure}[h]
\begin{center}
\begin{picture}(36,4.5)

\put(0,0){

\put(0,4){\p} \put(2,4){\p} \put(4,4){\p} \put(6,4){\p}
\put(0,2){\p} \put(2,2){\p} \put(4,2){\p} \put(6,2){\p}
                    \put(2,0){\p} \put(4,0){\p} \put(6,0){\p}

\put(-0.6,4.2){1}  \put(1.9,4.2){3} \put(3.9,4.2){1} \put(6.2,4.2){2}
\put(-0.6,2.2){2} \put(1.6,2.3){3} \put(4.2,2.2){2} \put(6.2,2){3}
                          \put(1.4,0.1){*}    \put(3.4,0.2){1} \put(6.2,0.2){4} 

\put(2,0){\line(1,0){4}}
\put(0,2){\line(1,0){6}}
\put(0,4){\line(1,0){6}}
\put(6,0){\line(0,1){4}}
\put(4,0){\line(0,1){4}}
\put(2,0){\line(0,1){2}}
\put(0,2){\line(0,1){2}}

\put(6,2){\line(-1,1){2}}
\put(2,0){\line(1,1){2}}
\put(6,0){\line(-1,1){2}}
\put(0,2){\line(1,1){2}}
\put(0,2){\line(2,1){4}}
\put(2,2){\line(1,1){2}}

}

\put(9,0){

\put(0,4){\p} \put(2,4){\p} \put(4,4){\p} \put(6,4){\p}
\put(0,2){\p} \put(2,2){\p} \put(4,2){\p} \put(6,2){\p}
                    \put(2,0){\p} \put(4,0){\p} \put(6,0){\p}

\put(-0.6,4.2){1}  \put(1.9,4.2){3} \put(3.9,4.2){1} \put(6.2,4.2){2}
\put(-0.6,2.2){2} \put(1.6,2.3){3} \put(4.2,2.2){2} \put(6.2,2){3}
                          \put(1.4,0.1){*}    \put(3.4,0.2){1} \put(6.2,0.2){4} 

\put(2,0){\line(1,0){4}}
\put(0,2){\line(1,0){6}}
\put(0,4){\line(1,0){6}}
\put(6,0){\line(0,1){4}}
\put(4,0){\line(0,1){4}}
\put(2,0){\line(0,1){2}}
\put(0,2){\line(0,1){2}}

\put(6,2){\line(-1,1){2}}
\put(2,0){\line(1,1){2}}
\put(6,0){\line(-1,1){2}}
\put(0,4){\line(1,-1){2}}
\put(0,4){\line(2,-1){4}}
\put(2,4){\line(1,-1){2}}

}

\put(18,0){

\put(0,4){\p} \put(2,4){\p} \put(4,4){\p} \put(6,4){\p}
\put(0,2){\p} \put(2,2){\p} \put(4,2){\p} \put(6,2){\p}
                    \put(2,0){\p} \put(4,0){\p} \put(6,0){\p}

\put(-0.6,4.2){?}  \put(1.9,4.2){?} \put(3.9,4.2){3} \put(6.2,4.2){1}
\put(-0.6,2.2){?} \put(1.6,2.1){*} \put(4.2,2.2){2} \put(6.2,2){3}
                          \put(1.4,0.2){?}    \put(3.4,0.2){1} \put(6.2,0.2){4} 

\put(2,0){\line(1,0){4}}
\put(0,2){\line(1,0){6}}
\put(0,4){\line(1,0){6}}
\put(6,0){\line(0,1){4}}
\put(4,0){\line(0,1){4}}
\put(2,0){\line(0,1){2}}
\put(0,2){\line(0,1){2}}

\put(4,2){\line(1,1){2}}
\put(4,0){\line(-1,1){2}}
\put(6,0){\line(-1,1){2}}
\put(0,2){\line(1,1){2}}
\put(0,2){\line(2,1){4}}
\put(2,2){\line(1,1){2}}

}

\put(27,0){

\put(0,4){\p} \put(2,4){\p} \put(4,4){\p} \put(6,4){\p}
\put(0,2){\p} \put(2,2){\p} \put(4,2){\p} \put(6,2){\p}
                    \put(2,0){\p} \put(4,0){\p} \put(6,0){\p}

\put(-0.6,4.2){*}  \put(1.9,4.2){1} \put(3.9,4.2){3} \put(6.2,4.2){1}
\put(-0.6,2.2){?} \put(1.6,2.3){3} \put(4.2,2.2){2} \put(6.2,2){3}
                          \put(1.4,0.2){2}    \put(3.4,0.2){1} \put(6.2,0.2){4} 

\put(2,0){\line(1,0){4}}
\put(0,2){\line(1,0){6}}
\put(0,4){\line(1,0){6}}
\put(6,0){\line(0,1){4}}
\put(4,0){\line(0,1){4}}
\put(2,0){\line(0,1){2}}
\put(0,2){\line(0,1){2}}

\put(4,2){\line(1,1){2}}
\put(4,0){\line(-1,1){2}}
\put(6,0){\line(-1,1){2}}
\put(0,4){\line(1,-1){2}}
\put(0,4){\line(2,-1){4}}
\put(2,4){\line(1,-1){2}}

}
\end{picture}
\caption{Four impossible subcases in the situation $S_1$ in Figure~\ref{conflict-situations}.} \label{S1-impossible-02}
\end{center}
\end{figure}

\begin{figure}[h]
\begin{center}
\begin{picture}(36,4.5)

\put(0,0){

\put(0,4){\p} \put(2,4){\p} \put(4,4){\p} \put(6,4){\p}
\put(0,2){\p} \put(2,2){\p} \put(4,2){\p} \put(6,2){\p}
\put(0,0){\p} \put(2,0){\p} \put(4,0){\p}

\put(-0.6,4.2){?}  \put(1.9,4.1){*} \put(3.9,4.2){?} \put(6.2,4.2){?}
\put(-0.6,2.2){1} \put(1.6,2.3){2} \put(3.8,2.2){3} \put(6.2,2){?}
 \put(-0.6,0.2){3}    \put(1.4,0.2){1} \put(4.2,0.2){4} 

\put(0,0){\line(1,0){4}}
\put(0,2){\line(1,0){6}}
\put(0,4){\line(1,0){6}}
\put(6,2){\line(0,1){2}}
\put(4,0){\line(0,1){2}}
\put(2,0){\line(0,1){4}}
\put(0,0){\line(0,1){4}}

\put(0,2){\line(1,1){2}}
\put(0,0){\line(1,1){2}}
\put(2,2){\line(1,-1){2}}
\put(2,4){\line(1,-1){2}}
\put(2,4){\line(2,-1){4}}
\put(4,4){\line(1,-1){2}}

}

\put(9,0){

\put(0,4){\p} \put(2,4){\p} \put(4,4){\p} \put(6,4){\p}
\put(0,2){\p} \put(2,2){\p} \put(4,2){\p} \put(6,2){\p}
\put(0,0){\p} \put(2,0){\p} \put(4,0){\p}

\put(-0.6,4.2){2}  \put(1.9,4.2){3} \put(3.9,4.2){1} \put(6.2,4.1){*}
\put(-0.6,2.2){1} \put(1.6,2.3){2} \put(3.8,2.2){3} \put(6.2,2){?}
 \put(-0.6,0.2){3}    \put(1.4,0.2){1} \put(4.2,0.2){4} 

\put(0,0){\line(1,0){4}}
\put(0,2){\line(1,0){6}}
\put(0,4){\line(1,0){6}}
\put(6,2){\line(0,1){2}}
\put(4,0){\line(0,1){2}}
\put(2,0){\line(0,1){4}}
\put(0,0){\line(0,1){4}}

\put(0,2){\line(1,1){2}}
\put(0,0){\line(1,1){2}}
\put(2,2){\line(1,-1){2}}
\put(2,2){\line(1,1){2}}
\put(2,2){\line(2,1){4}}
\put(4,2){\line(1,1){2}}

}

\put(18,0){

\put(0,4){\p} \put(2,4){\p} \put(4,4){\p} \put(6,4){\p}
\put(0,2){\p} \put(2,2){\p} \put(4,2){\p} \put(6,2){\p}
\put(0,0){\p} \put(2,0){\p} \put(4,0){\p}

\put(-0.6,4.2){1}  \put(1.9,4.2){3} \put(3.9,4.2){1} \put(6.2,4.1){*}
\put(-0.6,2.2){3} \put(1.6,2.3){2} \put(3.8,2.2){3} \put(6.2,2){?}
 \put(-0.6,0.2){2}    \put(1.4,0.2){1} \put(4.2,0.2){4} 

\put(0,0){\line(1,0){4}}
\put(0,2){\line(1,0){6}}
\put(0,4){\line(1,0){6}}
\put(6,2){\line(0,1){2}}
\put(4,0){\line(0,1){2}}
\put(2,0){\line(0,1){4}}
\put(0,0){\line(0,1){4}}

\put(0,4){\line(1,-1){2}}
\put(2,0){\line(-1,1){2}}
\put(2,2){\line(1,-1){2}}
\put(2,2){\line(1,1){2}}
\put(2,2){\line(2,1){4}}
\put(4,2){\line(1,1){2}}

}

\put(27,0){

\put(0,4){\p} \put(2,4){\p} \put(4,4){\p} \put(6,4){\p}
\put(0,2){\p} \put(2,2){\p} \put(4,2){\p} \put(6,2){\p}
\put(0,0){\p} \put(2,0){\p} \put(4,0){\p}

\put(-0.6,4.2){1}  \put(1.9,4.1){*} \put(3.9,4.2){?} \put(6.2,4.2){?}
\put(-0.6,2.2){3} \put(1.6,2.3){2} \put(3.8,2.2){3} \put(6.2,2){?}
 \put(-0.6,0.2){2}    \put(1.4,0.2){1} \put(4.2,0.2){4} 

\put(0,0){\line(1,0){4}}
\put(0,2){\line(1,0){6}}
\put(0,4){\line(1,0){6}}
\put(6,2){\line(0,1){2}}
\put(4,0){\line(0,1){2}}
\put(2,0){\line(0,1){4}}
\put(0,0){\line(0,1){4}}

\put(0,4){\line(1,-1){2}}
\put(2,0){\line(-1,1){2}}
\put(2,2){\line(1,-1){2}}
\put(2,4){\line(1,-1){2}}
\put(2,4){\line(2,-1){4}}
\put(4,4){\line(1,-1){2}}

}

\end{picture}
\caption{Four impossible subcases in the situation $S_1$ in Figure~\ref{conflict-situations}.} \label{S1-subcases-03}
\end{center}
\end{figure}

Moreover, the leftmost two graphs in Figure~\ref{S1-1-subcases} are, respectively, $B_2$ rotated 180 degrees and $B_1$ flipped with respect to a horizontal line. Finally, in each of the rightmost two graphs in Figure~\ref{S1-1-subcases}, we have a vertex labeled by * that would require colour 4 contradicting the fact that $v$ is supposed to be the only vertex coloured by 4. \\

\begin{figure}[h]
\begin{center}
\begin{picture}(36,4.5)

\put(0,0){

\put(0,4){\p} \put(2,4){\p} \put(4,4){\p} \put(6,4){\p}
\put(0,2){\p} \put(2,2){\p} \put(4,2){\p} \put(6,2){\p}
                    \put(2,0){\p} \put(4,0){\p} \put(6,0){\p}

\put(-0.6,4.2){3}  \put(1.9,4.2){1} \put(3.9,4.2){3} \put(6.2,4.2){1}
\put(-0.6,2.2){2} \put(1.6,2.3){1} \put(4.2,2.2){2} \put(6.2,2){3}
                          \put(1.4,0.2){3}    \put(3.4,0.2){1} \put(6.2,0.2){4} 

\put(2,0){\line(1,0){4}}
\put(0,2){\line(1,0){6}}
\put(0,4){\line(1,0){6}}
\put(6,0){\line(0,1){4}}
\put(4,0){\line(0,1){4}}
\put(2,0){\line(0,1){2}}
\put(0,2){\line(0,1){2}}

\put(2,0){\line(1,1){4}}
\put(6,0){\line(-1,1){2}}
\put(0,2){\line(1,1){2}}
\put(0,2){\line(2,1){4}}
\put(2,2){\line(1,1){2}}

}

\put(9,0){

\put(0,4){\p} \put(2,4){\p} \put(4,4){\p} \put(6,4){\p}
\put(0,2){\p} \put(2,2){\p} \put(4,2){\p} \put(6,2){\p}
                    \put(2,0){\p} \put(4,0){\p} \put(6,0){\p}

\put(-0.6,4.2){3}  \put(1.9,4.2){1} \put(3.9,4.2){3} \put(6.2,4.2){1}
\put(-0.6,2.2){2} \put(1.6,2.3){1} \put(4.2,2.2){2} \put(6.2,2){3}
                          \put(1.4,0.2){3}    \put(3.4,0.2){1} \put(6.2,0.2){4} 

\put(2,0){\line(1,0){4}}
\put(0,2){\line(1,0){6}}
\put(0,4){\line(1,0){6}}
\put(6,0){\line(0,1){4}}
\put(4,0){\line(0,1){4}}
\put(2,0){\line(0,1){2}}
\put(0,2){\line(0,1){2}}

\put(2,0){\line(1,1){4}}
\put(6,0){\line(-1,1){2}}
\put(0,4){\line(1,-1){2}}
\put(0,4){\line(2,-1){4}}
\put(2,4){\line(1,-1){2}}

}

\put(18,0){

\put(0,4){\p} \put(2,4){\p} \put(4,4){\p} \put(6,4){\p}
\put(0,2){\p} \put(2,2){\p} \put(4,2){\p} \put(6,2){\p}
\put(0,0){\p} \put(2,0){\p} \put(4,0){\p}

\put(-0.6,4.2){3}  \put(1.9,4.2){1} \put(3.9,4.2){3} \put(6.2,4.2){1}
\put(-0.6,2.2){1} \put(1.6,2.3){2} \put(3.8,2.2){3} \put(6.2,2){2}
 \put(-0.6,0.2){3}    \put(1.4,0.2){1} \put(4.2,0.2){4} 

\put(0,0){\line(1,0){4}}
\put(0,2){\line(1,0){6}}
\put(0,4){\line(1,0){6}}
\put(6,2){\line(0,1){2}}
\put(4,0){\line(0,1){2}}
\put(2,0){\line(0,1){4}}
\put(0,0){\line(0,1){4}}

\put(0,4){\line(1,-1){2}}
\put(0,0){\line(1,1){2}}
\put(2,2){\line(1,-1){2}}
\put(2,2){\line(1,1){2}}
\put(2,2){\line(2,1){4}}
\put(4,2){\line(1,1){2}}

}

\put(27,0){

\put(0,4){\p} \put(2,4){\p} \put(4,4){\p} \put(6,4){\p}
\put(0,2){\p} \put(2,2){\p} \put(4,2){\p} \put(6,2){\p}
\put(0,0){\p} \put(2,0){\p} \put(4,0){\p}

\put(-0.6,4.2){3}  \put(1.9,4.2){1} \put(3.9,4.2){3} \put(6.2,4.2){1}
\put(-0.6,2.2){1} \put(1.6,2.3){2} \put(3.8,2.2){3} \put(6.2,2){2}
 \put(-0.6,0.2){3}    \put(1.4,0.2){1} \put(4.2,0.2){4} 

\put(0,0){\line(1,0){4}}
\put(0,2){\line(1,0){6}}
\put(0,4){\line(1,0){6}}
\put(6,2){\line(0,1){2}}
\put(4,0){\line(0,1){2}}
\put(2,0){\line(0,1){4}}
\put(0,0){\line(0,1){4}}

\put(0,4){\line(1,-1){2}}
\put(0,0){\line(1,1){2}}
\put(2,2){\line(1,-1){2}}
\put(2,4){\line(1,-1){2}}
\put(2,4){\line(2,-1){4}}
\put(4,4){\line(1,-1){2}}

}

\end{picture}
\caption{Four subcases in the situation $S_1$ in Figure~\ref{conflict-situations}.} \label{S1-subcases}
\end{center}
\end{figure}
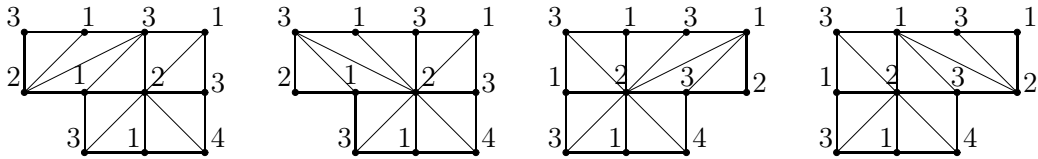

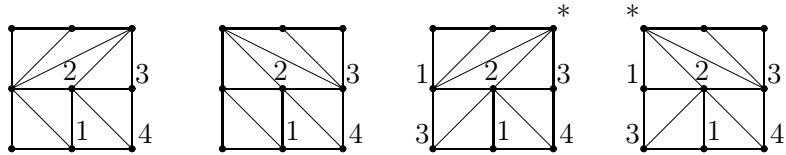
\begin{figure}[h]
\begin{center}
\begin{picture}(27,4)

\put(0,0){

\put(0,4){\p} \put(2,4){\p} \put(4,4){\p} 
\put(0,2){\p} \put(2,2){\p} \put(4,2){\p} 
\put(0,0){\p} \put(2,0){\p} \put(4,0){\p}

\put(1.7,2.3){2} \put(4.1,2.2){3} 
\put(2.1,0.3){1} \put(4.2,0.2){4} 

\put(0,0){\line(1,0){4}}
\put(0,2){\line(1,0){4}}
\put(0,4){\line(1,0){4}}
\put(4,0){\line(0,1){4}}
\put(2,0){\line(0,1){2}}
\put(0,0){\line(0,1){4}}

\put(4,4){\line(-1,-1){2}}
\put(0,2){\line(1,1){2}}
\put(0,2){\line(1,-1){2}}
\put(0,2){\line(2,1){4}}
\put(2,2){\line(1,-1){2}}

}

\put(7,0){

\put(0,4){\p} \put(2,4){\p} \put(4,4){\p} 
\put(0,2){\p} \put(2,2){\p} \put(4,2){\p} 
\put(0,0){\p} \put(2,0){\p} \put(4,0){\p}

\put(1.7,2.3){2} \put(4.1,2.2){3} 
\put(2.1,0.3){1} \put(4.2,0.2){4} 

\put(0,0){\line(1,0){4}}
\put(0,2){\line(1,0){4}}
\put(0,4){\line(1,0){4}}
\put(4,0){\line(0,1){4}}
\put(2,0){\line(0,1){2}}
\put(0,0){\line(0,1){4}}

\put(2,4){\line(1,-1){2}}
\put(2,2){\line(-1,1){2}}
\put(0,2){\line(1,-1){2}}
\put(0,4){\line(2,-1){4}}
\put(2,2){\line(1,-1){2}}

}

\put(14,0){

\put(0,4){\p} \put(2,4){\p} \put(4,4){\p} 
\put(0,2){\p} \put(2,2){\p} \put(4,2){\p} 
\put(0,0){\p} \put(2,0){\p} \put(4,0){\p}

                                                   \put(4.1,4.1){*} 
\put(-0.6,2.2){1}  \put(1.7,2.3){2} \put(4.1,2.2){3} 
 \put(-0.6,0.2){3}  \put(2.1,0.3){1} \put(4.2,0.2){4} 

\put(0,0){\line(1,0){4}}
\put(0,2){\line(1,0){4}}
\put(0,4){\line(1,0){4}}
\put(4,0){\line(0,1){4}}
\put(2,0){\line(0,1){2}}
\put(0,0){\line(0,1){4}}

\put(2,2){\line(1,1){2}}
\put(0,2){\line(1,1){2}}
\put(0,0){\line(1,1){2}}
\put(0,2){\line(2,1){4}}
\put(2,2){\line(1,-1){2}}

}

\put(21,0){

\put(0,4){\p} \put(2,4){\p} \put(4,4){\p} 
\put(0,2){\p} \put(2,2){\p} \put(4,2){\p} 
\put(0,0){\p} \put(2,0){\p} \put(4,0){\p}

\put(-0.6,4.1){*}  
\put(-0.6,2.2){1}  \put(1.7,2.3){2} \put(4.1,2.2){3} 
 \put(-0.6,0.2){3}  \put(2.1,0.3){1} \put(4.2,0.2){4} 

\put(0,0){\line(1,0){4}}
\put(0,2){\line(1,0){4}}
\put(0,4){\line(1,0){4}}
\put(4,0){\line(0,1){4}}
\put(2,0){\line(0,1){2}}
\put(0,0){\line(0,1){4}}

\put(4,2){\line(-1,1){2}}
\put(2,2){\line(-1,1){2}}
\put(0,0){\line(1,1){2}}
\put(0,4){\line(2,-1){4}}
\put(2,2){\line(1,-1){2}}

}

\end{picture}
\caption{Four more subcases in the situation $S_1$ in Figure~\ref{conflict-situations}.} \label{S1-1-subcases}
\end{center}
\end{figure}

\noindent 
{\bf Situation $S_2$.} Note that in the case when  the  vertex $v$ is involved in a subgraph presented schematically to the left in Figure~\ref{S2-impossible} (the question marks there indicate that triangulations of respective squares are unknown to us), such a subgraph must be either $T_1$ or $T_2$. Indeed, otherwise, the subgraph must be one of the two graphs presented in Figure~\ref{S2-impossible}  to the right of the leftmost graph. However, in each of the two cases, we have a vertex labeled by * that would require colour 4 contradicting the fact that $v$ is supposed to be the only vertex coloured by 4. Similarly, the subcases presented in Figure~\ref{S2-2-impossible} are impossible since they contain a vertex, labeled by *, that requires colour 4.

\begin{figure}[h]
\begin{center}
\begin{picture}(20,4.5)

\put(0,0){

\put(0,4){\p} \put(2,4){\p} \put(4,4){\p} 
\put(0,2){\p} \put(2,2){\p} \put(4,2){\p} 
\put(0,0){\p} \put(2,0){\p} \put(4,0){\p}

\put(-0.6,2.2){2}  \put(1.4,2.3){1} \put(4.1,2.2){3} 
 \put(-0.6,0.2){1}  \put(1.4,0.2){4} \put(4.2,0.2){?} 

\put(0,0){\line(1,0){4}}
\put(0,2){\line(1,0){4}}
\put(0,4){\line(1,0){4}}
\put(4,0){\line(0,1){4}}
\put(2,0){\line(0,1){4}}
\put(0,0){\line(0,1){4}}

\put(2,0){\line(1,1){2}}
\put(2,0){\line(-1,1){2}}

\put(0.5,3){{\large ?}}
\put(3,3){{\large ?}}

}

\put(7,0){

\put(0,4){\p} \put(2,4){\p} \put(4,4){\p} 
\put(0,2){\p} \put(2,2){\p} \put(4,2){\p} 
\put(0,0){\p} \put(2,0){\p} \put(4,0){\p}

\put(-0.6,4.1){?}  \put(1.8,4.1){*} \put(4.1,4.2){?} 
\put(-0.6,2.2){2}  \put(1.4,2.3){1} \put(4.1,2.2){3} 
 \put(-0.6,0.2){1}  \put(1.4,0.2){4} \put(4.2,0.2){?} 

\put(0,0){\line(1,0){4}}
\put(0,2){\line(1,0){4}}
\put(0,4){\line(1,0){4}}
\put(4,0){\line(0,1){4}}
\put(2,0){\line(0,1){4}}
\put(0,0){\line(0,1){4}}

\put(2,0){\line(1,1){2}}
\put(2,0){\line(-1,1){2}}
\put(2,4){\line(-1,-1){2}}
\put(2,4){\line(1,-1){2}}

}

\put(14,0){

\put(0,4){\p} \put(2,4){\p} \put(4,4){\p} 
\put(0,2){\p} \put(2,2){\p} \put(4,2){\p} 
\put(0,0){\p} \put(2,0){\p} \put(4,0){\p}

\put(-0.6,4.2){3}  \put(1.9,4.2){2} \put(4.1,4.1){*} 
\put(-0.6,2.2){2}  \put(1.4,2.3){1} \put(4.1,2.2){3} 
 \put(-0.6,0.2){1}  \put(1.4,0.2){4} \put(4.2,0.2){?} 

\put(0,0){\line(1,0){4}}
\put(0,2){\line(1,0){4}}
\put(0,4){\line(1,0){4}}
\put(4,0){\line(0,1){4}}
\put(2,0){\line(0,1){4}}
\put(0,0){\line(0,1){4}}

\put(2,0){\line(1,1){2}}
\put(2,0){\line(-1,1){2}}
\put(2,2){\line(1,1){2}}
\put(2,2){\line(-1,1){2}}

}

\end{picture}
\caption{Impossible subcases in the situation $S_2$ in Figure~\ref{conflict-situations}.} \label{S2-impossible}
\end{center}
\end{figure}

\begin{figure}[h]
\begin{center}
\begin{picture}(13,4)

\put(0,0){

\put(0,4){\p} \put(2,4){\p} \put(4,4){\p} 
\put(0,2){\p} \put(2,2){\p} \put(4,2){\p} 
\put(0,0){\p} \put(2,0){\p} \put(4,0){\p}

\put(-0.6,4.2){?}  \put(1.8,4.2){?} \put(4.1,4.1){*} 
\put(-0.6,2.2){2}  \put(1.4,2.3){1} \put(4.1,2.2){3} 
 \put(-0.6,0.2){1}  \put(1.4,0.2){4} \put(4.2,0.2){?} 

\put(0,0){\line(1,0){4}}
\put(0,2){\line(1,0){4}}
\put(0,4){\line(1,0){4}}
\put(4,0){\line(0,1){4}}
\put(2,0){\line(0,1){2}}
\put(0,0){\line(0,1){4}}

\put(0,2){\line(1,-1){2}}
\put(4,2){\line(-1,-1){2}}
\put(0,2){\line(1,1){2}}
\put(0,2){\line(2,1){4}}
\put(2,2){\line(1,1){2}}

}

\put(7,0){

\put(0,4){\p} \put(2,4){\p} \put(4,4){\p} 
\put(0,2){\p} \put(2,2){\p} \put(4,2){\p} 
\put(0,0){\p} \put(2,0){\p} \put(4,0){\p}

\put(-0.6,4.2){*}  \put(1.8,4.2){?} \put(4.1,4.1){?} 
\put(-0.6,2.2){2}  \put(1.4,2.3){1} \put(4.1,2.2){3} 
 \put(-0.6,0.2){1}  \put(1.4,0.2){4} \put(4.2,0.2){?} 

\put(0,0){\line(1,0){4}}
\put(0,2){\line(1,0){4}}
\put(0,4){\line(1,0){4}}
\put(4,0){\line(0,1){4}}
\put(2,0){\line(0,1){2}}
\put(0,0){\line(0,1){4}}

\put(0,2){\line(1,-1){2}}
\put(4,2){\line(-1,-1){2}}
\put(2,2){\line(-1,1){2}}
\put(4,2){\line(-2,1){4}}
\put(4,2){\line(-1,1){2}}

}

\end{picture}
\caption{Two more impossible subcases in the situation $S_2$ in Figure~\ref{conflict-situations}.} \label{S2-2-impossible}
\end{center}
\end{figure}
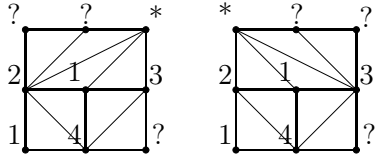

But then we have four possible subcases to be considered, which are presented in Figure~\ref{S2-subcases}, where colouring of vertices is shown. However, the graphs in Figure~\ref{S2-subcases}, from left to right, are, respectively:
\begin{itemize} 
\item $A_4$ flipped with respect to a horizontal line;
\item $A_2$ flipped with respect to a horizontal line;
\item $A_4$ rotated 180 degrees;
\item $A_2$ rotated 180 degrees.
\end{itemize}

\begin{figure}[h]
\begin{center}
\begin{picture}(36,4)

\put(0,0){

\put(0,4){\p} \put(2,4){\p} \put(4,4){\p} \put(6,4){\p}
\put(0,2){\p} \put(2,2){\p} \put(4,2){\p} \put(6,2){\p}
                    \put(2,0){\p} \put(4,0){\p} \put(6,0){\p}

\put(-0.6,4.2){3}  \put(1.9,4.2){2} \put(3.9,4.2){3} \put(6.2,4.2){2}
\put(-0.6,2.2){1} \put(1.6,2.3){2} \put(4.2,2.2){1} \put(6.2,2){3}
                          \put(1.4,0.2){1}    \put(3.4,0.2){4} \put(6.2,0.2){?} 

\put(2,0){\line(1,0){4}}
\put(0,2){\line(1,0){6}}
\put(0,4){\line(1,0){6}}
\put(6,0){\line(0,1){4}}
\put(4,0){\line(0,1){4}}
\put(2,0){\line(0,1){2}}
\put(0,2){\line(0,1){2}}

\put(4,2){\line(1,1){2}}
\put(6,2){\line(-1,-1){2}}
\put(0,2){\line(1,1){2}}
\put(0,2){\line(2,1){4}}
\put(2,2){\line(1,1){2}}
\put(2,2){\line(1,-1){2}}

}

\put(9,0){

\put(0,4){\p} \put(2,4){\p} \put(4,4){\p} \put(6,4){\p}
\put(0,2){\p} \put(2,2){\p} \put(4,2){\p} \put(6,2){\p}
                    \put(2,0){\p} \put(4,0){\p} \put(6,0){\p}

\put(-0.6,4.2){3}  \put(1.9,4.2){2} \put(3.9,4.2){3} \put(6.2,4.2){2}
\put(-0.6,2.2){1} \put(1.6,2.3){2} \put(4.2,2.2){1} \put(6.2,2){3}
                          \put(1.4,0.2){1}    \put(3.4,0.2){4} \put(6.2,0.2){?} 

\put(2,0){\line(1,0){4}}
\put(0,2){\line(1,0){6}}
\put(0,4){\line(1,0){6}}
\put(6,0){\line(0,1){4}}
\put(4,0){\line(0,1){4}}
\put(2,0){\line(0,1){2}}
\put(0,2){\line(0,1){2}}

\put(4,2){\line(1,1){2}}
\put(6,2){\line(-1,-1){2}}
\put(0,4){\line(1,-1){2}}
\put(0,4){\line(2,-1){4}}
\put(2,4){\line(1,-1){2}}
\put(2,2){\line(1,-1){2}}

}

\put(18,0){

\put(0,4){\p} \put(2,4){\p} \put(4,4){\p} \put(6,4){\p}
\put(0,2){\p} \put(2,2){\p} \put(4,2){\p} \put(6,2){\p}
\put(0,0){\p} \put(2,0){\p} \put(4,0){\p}

\put(-0.6,4.2){3}  \put(1.9,4.2){2} \put(3.9,4.2){3} \put(6.2,4.2){2}
\put(-0.6,2.2){2} \put(1.6,2.3){1} \put(3.8,2.2){3} \put(6.2,2){1}
 \put(-0.6,0.2){1}    \put(1.4,0.2){4} \put(4.2,0.2){?} 

\put(0,0){\line(1,0){4}}
\put(0,2){\line(1,0){6}}
\put(0,4){\line(1,0){6}}
\put(6,2){\line(0,1){2}}
\put(4,0){\line(0,1){2}}
\put(2,0){\line(0,1){4}}
\put(0,0){\line(0,1){4}}

\put(0,4){\line(1,-1){2}}
\put(0,2){\line(1,-1){2}}
\put(2,0){\line(1,1){2}}
\put(2,4){\line(1,-1){2}}
\put(2,4){\line(2,-1){4}}
\put(4,4){\line(1,-1){2}}

}

\put(27,0){

\put(0,4){\p} \put(2,4){\p} \put(4,4){\p} \put(6,4){\p}
\put(0,2){\p} \put(2,2){\p} \put(4,2){\p} \put(6,2){\p}
\put(0,0){\p} \put(2,0){\p} \put(4,0){\p}

\put(-0.6,4.2){3}  \put(1.9,4.2){2} \put(3.9,4.2){3} \put(6.2,4.2){2}
\put(-0.6,2.2){2} \put(1.6,2.3){1} \put(3.8,2.2){3} \put(6.2,2){1}
 \put(-0.6,0.2){1}    \put(1.4,0.2){4} \put(4.2,0.2){?} 

\put(0,0){\line(1,0){4}}
\put(0,2){\line(1,0){6}}
\put(0,4){\line(1,0){6}}
\put(6,2){\line(0,1){2}}
\put(4,0){\line(0,1){2}}
\put(2,0){\line(0,1){4}}
\put(0,0){\line(0,1){4}}

\put(0,4){\line(1,-1){2}}
\put(0,2){\line(1,-1){2}}
\put(2,2){\line(1,1){2}}
\put(2,2){\line(2,1){4}}
\put(2,0){\line(1,1){4}}

}

\end{picture}
\caption{Four subcases in the situation $S_2$ in Figure~\ref{conflict-situations}.} \label{S2-subcases}
\end{center}
\end{figure}
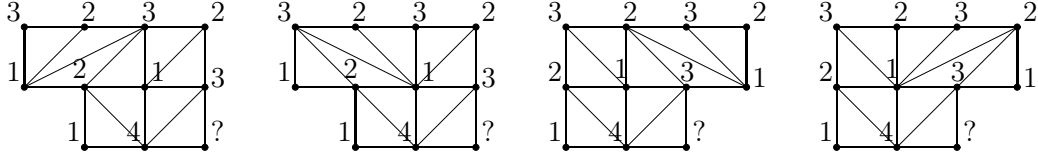

\noindent 
{\bf Situation $S_3$.} Note that in the case when  the  vertex $v$ is involved in a subgraph presented schematically to the left in Figure~\ref{S3-impossible} (the question marks there indicate that triangulations of respective squares are unknown to us), such a subgraph must be either $T_1$ or $T_2$. Indeed, otherwise, the subgraph must be one of the two graphs presented in Figure~\ref{S3-impossible}  to the right of the leftmost graph. However, in each of the two cases, we have a vertex labeled by * that would require colour 4 contradicting the fact that $v$ is supposed to be the only vertex coloured by 4. 

\begin{figure}[h]
\begin{center}
\begin{picture}(20,4)

\put(0,0){

\put(0,4){\p} \put(2,4){\p} \put(4,4){\p} 
\put(0,2){\p} \put(2,2){\p} \put(4,2){\p} 
\put(0,0){\p} \put(2,0){\p} \put(4,0){\p}

\put(-0.6,2.2){3}  \put(1.4,2.3){2} \put(4.1,2.2){3} 
 \put(-0.6,0.2){1}  \put(1.4,0.2){4} \put(4.2,0.2){?} 

\put(0,0){\line(1,0){4}}
\put(0,2){\line(1,0){4}}
\put(0,4){\line(1,0){4}}
\put(4,0){\line(0,1){4}}
\put(2,0){\line(0,1){4}}
\put(0,0){\line(0,1){4}}

\put(2,0){\line(1,1){2}}
\put(0,0){\line(1,1){2}}

\put(0.5,3){{\large ?}}
\put(3,3){{\large ?}}

}

\put(7,0){

\put(0,4){\p} \put(2,4){\p} \put(4,4){\p} 
\put(0,2){\p} \put(2,2){\p} \put(4,2){\p} 
\put(0,0){\p} \put(2,0){\p} \put(4,0){\p}

\put(-0.6,4.2){2}  \put(1.8,4.2){1} \put(4.1,4.1){*} 
\put(-0.6,2.2){3}  \put(1.4,2.3){2} \put(4.1,2.2){3} 
 \put(-0.6,0.2){1}  \put(1.4,0.2){4} \put(4.2,0.2){?} 

\put(0,0){\line(1,0){4}}
\put(0,2){\line(1,0){4}}
\put(0,4){\line(1,0){4}}
\put(4,0){\line(0,1){4}}
\put(2,0){\line(0,1){4}}
\put(0,0){\line(0,1){4}}

\put(2,0){\line(1,1){2}}
\put(0,0){\line(1,1){2}}
\put(2,2){\line(1,1){2}}
\put(0,2){\line(1,1){2}}

}

\put(14,0){

\put(0,4){\p} \put(2,4){\p} \put(4,4){\p} 
\put(0,2){\p} \put(2,2){\p} \put(4,2){\p} 
\put(0,0){\p} \put(2,0){\p} \put(4,0){\p}

\put(-0.6,4.2){1}  \put(1.9,4.1){*} \put(4.1,4.1){?} 
\put(-0.6,2.2){3}  \put(1.4,2.3){2} \put(4.1,2.2){3} 
 \put(-0.6,0.2){1}  \put(1.4,0.2){4} \put(4.2,0.2){?} 

\put(0,0){\line(1,0){4}}
\put(0,2){\line(1,0){4}}
\put(0,4){\line(1,0){4}}
\put(4,0){\line(0,1){4}}
\put(2,0){\line(0,1){4}}
\put(0,0){\line(0,1){4}}

\put(0,4){\line(1,-1){2}}
\put(2,4){\line(1,-1){2}}
\put(2,0){\line(1,1){2}}
\put(0,0){\line(1,1){2}}

}

\end{picture}
\caption{Impossible subcases in the situation $S_3$ in Figure~\ref{conflict-situations}.} \label{S3-impossible}
\end{center}
\end{figure}

But then we have six possible subcases to be considered: Four subcases are presented in Figure~\ref{S3-subcases}, where colouring of vertices is shown, and two subcases correspond to $B_1$ rotated 180 degrees, and $B_2$ flipped with respect to a horizontal line; $B_1$ and $B_2$ are presented in Figure~\ref{non-repr-minim-triang}, where the colours of vertices are omitted (the vertex $v$ corresponds to the middle vertex on the top row in each of the graphs). However, the graphs in Figure~\ref{S3-subcases} are, respectively, $A_7$, $A_8$, $A_6$ and $A_5$ flipped with respect to a vertical line.\\

\begin{figure}[h]
\begin{center}
\begin{picture}(35,4.5)

\put(0,0){

\put(0,4){\p} \put(2,4){\p} \put(4,4){\p} \put(6,4){\p}
\put(0,2){\p} \put(2,2){\p} \put(4,2){\p} \put(6,2){\p}
                     \put(2,0){\p} \put(4,0){\p} \put(6,0){\p}

\put(-0.6,4.2){1}  \put(1.9,4.2){3} \put(3.9,4.2){1} \put(6.2,4.2){2}
\put(-0.6,2.2){2} \put(1.7,2.2){3} \put(4.1,2.2){2} \put(6.2,2){3}
                         \put(1.4,0.2){1}    \put(3.4,0.2){4} \put(6.2,0.2){?} 

\put(2,0){\line(1,0){4}}
\put(0,2){\line(1,0){6}}
\put(0,4){\line(1,0){6}}
\put(6,0){\line(0,1){4}}
\put(4,0){\line(0,1){4}}
\put(2,0){\line(0,1){2}}
\put(0,2){\line(0,1){2}}

\put(2,0){\line(1,1){2}}
\put(4,0){\line(1,1){2}}
\put(4,4){\line(1,-1){2}}
\put(0,4){\line(1,-1){2}}
\put(0,4){\line(2,-1){4}}
\put(2,4){\line(1,-1){2}}

}

\put(9,0){

\put(0,4){\p} \put(2,4){\p} \put(4,4){\p} \put(6,4){\p}
\put(0,2){\p} \put(2,2){\p} \put(4,2){\p} \put(6,2){\p}
                     \put(2,0){\p} \put(4,0){\p} \put(6,0){\p}

\put(-0.6,4.2){1}  \put(1.9,4.2){3} \put(3.9,4.2){1} \put(6.2,4.2){2}
\put(-0.6,2.2){2} \put(1.6,2.2){3} \put(4.2,2.2){2} \put(6.2,2){3}
                         \put(1.4,0.2){1}    \put(3.4,0.2){4} \put(6.2,0.2){?} 

\put(2,0){\line(1,0){4}}
\put(0,2){\line(1,0){6}}
\put(0,4){\line(1,0){6}}
\put(6,0){\line(0,1){4}}
\put(4,0){\line(0,1){4}}
\put(2,0){\line(0,1){2}}
\put(0,2){\line(0,1){2}}

\put(2,0){\line(1,1){2}}
\put(4,0){\line(1,1){2}}
\put(4,4){\line(1,-1){2}}
\put(0,2){\line(1,1){2}}
\put(0,2){\line(2,1){4}}
\put(2,2){\line(1,1){2}}

}

\put(18,0){

\put(0,4){\p} \put(2,4){\p} \put(4,4){\p} \put(6,4){\p}
\put(0,2){\p} \put(2,2){\p} \put(4,2){\p} \put(6,2){\p}
\put(0,0){\p} \put(2,0){\p} \put(4,0){\p}

\put(-0.6,4.2){2}  \put(1.9,4.2){1} \put(3.9,4.2){3} \put(6.2,4.2){1}
\put(-0.6,2.2){3} \put(1.5,2.3){2} \put(3.7,2.2){3} \put(6.2,2){2}
 \put(-0.6,0.2){1}    \put(1.4,0.2){4} \put(4.2,0.2){?} 

\put(0,0){\line(1,0){4}}
\put(0,2){\line(1,0){6}}
\put(0,4){\line(1,0){6}}
\put(6,2){\line(0,1){2}}
\put(4,0){\line(0,1){2}}
\put(2,0){\line(0,1){4}}
\put(0,0){\line(0,1){4}}

\put(0,2){\line(1,1){2}}
\put(0,0){\line(1,1){2}}
\put(2,0){\line(1,1){2}}
\put(2,2){\line(1,1){2}}
\put(2,2){\line(2,1){4}}
\put(4,2){\line(1,1){2}}

}

\put(27,0){

\put(0,4){\p} \put(2,4){\p} \put(4,4){\p} \put(6,4){\p}
\put(0,2){\p} \put(2,2){\p} \put(4,2){\p} \put(6,2){\p}
\put(0,0){\p} \put(2,0){\p} \put(4,0){\p}

\put(-0.6,4.2){2}  \put(1.9,4.2){1} \put(3.9,4.2){3} \put(6.2,4.2){1}
\put(-0.6,2.2){3} \put(1.5,2.3){2} \put(3.7,2.2){3} \put(6.2,2){2}
 \put(-0.6,0.2){1}    \put(1.4,0.2){4} \put(4.2,0.2){?} 

\put(0,0){\line(1,0){4}}
\put(0,2){\line(1,0){6}}
\put(0,4){\line(1,0){6}}
\put(6,2){\line(0,1){2}}
\put(4,0){\line(0,1){2}}
\put(2,0){\line(0,1){4}}
\put(0,0){\line(0,1){4}}

\put(0,2){\line(1,1){2}}
\put(0,0){\line(1,1){2}}
\put(2,0){\line(1,1){2}}
\put(2,4){\line(1,-1){2}}
\put(2,4){\line(2,-1){4}}
\put(4,4){\line(1,-1){2}}

}

\end{picture}
\caption{Four subcases of the situation $S_3$ in Figure~\ref{conflict-situations}.} \label{S3-subcases}
\end{center}
\end{figure}
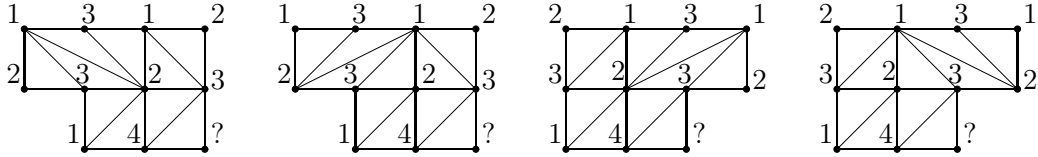

\noindent 
{\bf Situation $S_4$.} In this case, we have only two subcases, namely, $B_1$ and $B_2$ in Figure~\ref{non-repr-minim-triang}. Indeed, two other situations presented in Figure~\ref{S4-impossible} are impossible (the vertices labeled by * there require colour 4 contradicting our choice of $v$).\\

\begin{figure}[h]
\begin{center}
\begin{picture}(13,4)

\put(0,0){

\put(0,4){\p} \put(2,4){\p} \put(4,4){\p} 
\put(0,2){\p} \put(2,2){\p} \put(4,2){\p} 
\put(0,0){\p} \put(2,0){\p} \put(4,0){\p}

\put(-0.6,4.2){?}  \put(1.8,4.2){*} \put(4.1,4.1){?} 
\put(-0.6,2.2){3}  \put(1.4,2.3){1} \put(4.1,2.2){2} 
 \put(-0.6,0.2){2}  \put(1.7,0.2){1} \put(4.2,0.2){4} 

\put(0,0){\line(1,0){4}}
\put(0,2){\line(1,0){4}}
\put(0,4){\line(1,0){4}}
\put(4,0){\line(0,1){4}}
\put(2,2){\line(0,1){2}}
\put(0,0){\line(0,1){4}}

\put(2,4){\line(1,-1){2}}
\put(0,2){\line(1,1){2}}
\put(0,2){\line(1,-1){2}}
\put(0,2){\line(2,-1){4}}
\put(2,2){\line(1,-1){2}}

}

\put(7,0){

\put(0,4){\p} \put(2,4){\p} \put(4,4){\p} 
\put(0,2){\p} \put(2,2){\p} \put(4,2){\p} 
\put(0,0){\p} \put(2,0){\p} \put(4,0){\p}

\put(-0.6,4.2){2}  \put(1.8,4.2){3} \put(4.1,4.1){*} 
\put(-0.6,2.2){3}  \put(1.4,2.3){1} \put(4.1,2.2){2} 
 \put(-0.6,0.2){2}  \put(1.7,0.2){1} \put(4.2,0.2){4} 

\put(0,0){\line(1,0){4}}
\put(0,2){\line(1,0){4}}
\put(0,4){\line(1,0){4}}
\put(4,0){\line(0,1){4}}
\put(2,2){\line(0,1){2}}
\put(0,0){\line(0,1){4}}

\put(2,2){\line(1,1){2}}
\put(2,2){\line(-1,1){2}}
\put(0,2){\line(1,-1){2}}
\put(0,2){\line(2,-1){4}}
\put(2,2){\line(1,-1){2}}

}

\end{picture}
\caption{Impossible subcases in the situation $S_4$ in Figure~\ref{conflict-situations}.} \label{S4-impossible}
\end{center}
\end{figure}
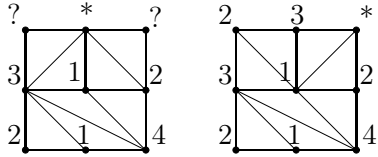

\noindent 
{\bf Situation $S_5$.} In this case, we have two possible and two impossible subcases presented in Figure~\ref{S5-imp-subcases}. The rightmost two graphs in that figure are impossible because the vertices labeled by * require colour 4. On the other hand, the leftmost two graphs are 3-colourable, which forces us to consider their extensions, namely, larger subgraphs in the situation $S_5$.

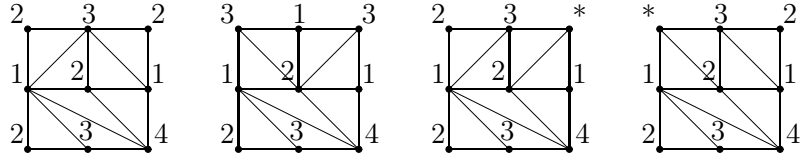
\begin{figure}[h]
\begin{center}
\begin{picture}(27,4)

\put(0,0){

\put(0,4){\p} \put(2,4){\p} \put(4,4){\p} 
\put(0,2){\p} \put(2,2){\p} \put(4,2){\p} 
\put(0,0){\p} \put(2,0){\p} \put(4,0){\p}

\put(-0.6,4.2){2}  \put(1.8,4.2){3} \put(4.1,4.2){2} 
\put(-0.6,2.2){1}  \put(1.4,2.3){2} \put(4.1,2.2){1} 
 \put(-0.6,0.2){2}  \put(1.7,0.3){3} \put(4.2,0.2){4} 

\put(0,0){\line(1,0){4}}
\put(0,2){\line(1,0){4}}
\put(0,4){\line(1,0){4}}
\put(4,0){\line(0,1){4}}
\put(2,2){\line(0,1){2}}
\put(0,0){\line(0,1){4}}

\put(2,4){\line(1,-1){2}}
\put(0,2){\line(1,1){2}}
\put(0,2){\line(1,-1){2}}
\put(0,2){\line(2,-1){4}}
\put(2,2){\line(1,-1){2}}

}

\put(7,0){

\put(0,4){\p} \put(2,4){\p} \put(4,4){\p} 
\put(0,2){\p} \put(2,2){\p} \put(4,2){\p} 
\put(0,0){\p} \put(2,0){\p} \put(4,0){\p}

\put(-0.6,4.2){3}  \put(1.8,4.2){1} \put(4.1,4.2){3} 
\put(-0.6,2.2){1}  \put(1.4,2.3){2} \put(4.1,2.2){1} 
 \put(-0.6,0.2){2}  \put(1.7,0.3){3} \put(4.2,0.2){4} 

\put(0,0){\line(1,0){4}}
\put(0,2){\line(1,0){4}}
\put(0,4){\line(1,0){4}}
\put(4,0){\line(0,1){4}}
\put(2,2){\line(0,1){2}}
\put(0,0){\line(0,1){4}}

\put(2,2){\line(1,1){2}}
\put(2,2){\line(-1,1){2}}
\put(0,2){\line(1,-1){2}}
\put(0,2){\line(2,-1){4}}
\put(2,2){\line(1,-1){2}}

}

\put(14,0){

\put(0,4){\p} \put(2,4){\p} \put(4,4){\p} 
\put(0,2){\p} \put(2,2){\p} \put(4,2){\p} 
\put(0,0){\p} \put(2,0){\p} \put(4,0){\p}

\put(-0.6,4.2){2}  \put(1.8,4.2){3} \put(4.1,4.1){*} 
\put(-0.6,2.2){1}  \put(1.4,2.3){2} \put(4.1,2.2){1} 
 \put(-0.6,0.2){2}  \put(1.7,0.3){3} \put(4.2,0.2){4} 

\put(0,0){\line(1,0){4}}
\put(0,2){\line(1,0){4}}
\put(0,4){\line(1,0){4}}
\put(4,0){\line(0,1){4}}
\put(2,2){\line(0,1){2}}
\put(0,0){\line(0,1){4}}

\put(2,2){\line(1,1){2}}
\put(0,2){\line(1,1){2}}
\put(0,2){\line(1,-1){2}}
\put(0,2){\line(2,-1){4}}
\put(2,2){\line(1,-1){2}}

}

\put(21,0){

\put(0,4){\p} \put(2,4){\p} \put(4,4){\p} 
\put(0,2){\p} \put(2,2){\p} \put(4,2){\p} 
\put(0,0){\p} \put(2,0){\p} \put(4,0){\p}

\put(-0.6,4.1){*}  \put(1.8,4.2){3} \put(4.1,4.2){2} 
\put(-0.6,2.2){1}  \put(1.4,2.3){2} \put(4.1,2.2){1} 
 \put(-0.6,0.2){2}  \put(1.7,0.3){3} \put(4.2,0.2){4} 

\put(0,0){\line(1,0){4}}
\put(0,2){\line(1,0){4}}
\put(0,4){\line(1,0){4}}
\put(4,0){\line(0,1){4}}
\put(2,2){\line(0,1){2}}
\put(0,0){\line(0,1){4}}

\put(4,2){\line(-1,1){2}}
\put(2,2){\line(-1,1){2}}
\put(0,2){\line(1,-1){2}}
\put(0,2){\line(2,-1){4}}
\put(2,2){\line(1,-1){2}}

}

\end{picture}
\caption{Possible and impossible subcases in the situation $S_5$ in Figure~\ref{conflict-situations}.} \label{S5-imp-subcases}
\end{center}
\end{figure}

Note that if there would be no other vertices to the left of the leftmost two graphs in Figure~\ref{S5-imp-subcases}, we could swap colours 2 and 3 in the bottom row to see that usage of colour 4 for $v$ is unnecessary. Thus, we can consider extensions of these graphs to the left. The leftmost graph in Figure~\ref{S5-imp-subcases} has two possible extensions recorded as the two leftmost graphs in Figure~\ref{S5-possible} (colours are omitted in that figure), and two impossible extensions (because of the issue with using colour 4 more than once indicated by *) --- see the leftmost two graphs in Figure~\ref{S5-impossible}.  Finally, next to the leftmost graph in Figure~\ref{S5-imp-subcases} has two possible extensions recorded as the two rightmost graphs in Figure~\ref{S5-possible} (colours are omitted in that figure), and two impossible extensions (because of the issue with using colour 4 more than once indicated by *) --- see the rightmost two graphs in Figure~\ref{S5-impossible}. However, the graphs in Figure~\ref{S5-possible}, from left to right, contain, respectively, the following graphs from Figure~\ref{non-repr-minim-triang} as induced subgraphs: 
\begin{itemize}
\item $A_7$ flipped with respect to a horizontal line;
\item $A_1$ flipped with respect to a vertical line;
\item $A_2$ flipped with respect to a vertical line;
\item $A_6$ rotated 180 degrees.
\end{itemize}
 
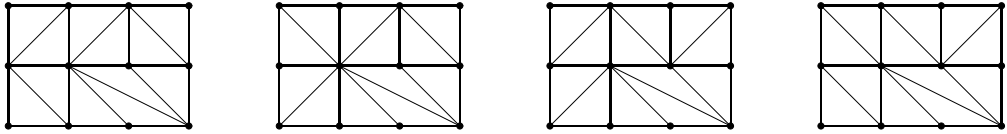
\begin{figure}[h]
\begin{center}
\begin{picture}(35,4.5)

\put(0,0){

\put(0,4){\p} \put(2,4){\p} \put(4,4){\p} \put(6,4){\p}
\put(0,2){\p} \put(2,2){\p} \put(4,2){\p} \put(6,2){\p}
\put(0,0){\p}  \put(2,0){\p} \put(4,0){\p} \put(6,0){\p}

\put(0,0){\line(1,0){6}}
\put(0,2){\line(1,0){6}}
\put(0,4){\line(1,0){6}}
\put(0,0){\line(0,1){4}}
\put(2,0){\line(0,1){4}}
\put(6,0){\line(0,1){4}}
\put(4,2){\line(0,1){2}}

\put(2,2){\line(1,-1){2}}
\put(2,2){\line(2,-1){4}}
\put(4,2){\line(1,-1){2}}

\put(0,2){\line(1,-1){2}}
\put(0,2){\line(1,1){2}}
\put(2,2){\line(1,1){2}}
\put(4,4){\line(1,-1){2}}

}

\put(9,0){

\put(0,4){\p} \put(2,4){\p} \put(4,4){\p} \put(6,4){\p}
\put(0,2){\p} \put(2,2){\p} \put(4,2){\p} \put(6,2){\p}
\put(0,0){\p}  \put(2,0){\p} \put(4,0){\p} \put(6,0){\p}

\put(0,0){\line(1,0){6}}
\put(0,2){\line(1,0){6}}
\put(0,4){\line(1,0){6}}
\put(0,0){\line(0,1){4}}
\put(2,0){\line(0,1){4}}
\put(6,0){\line(0,1){4}}
\put(4,2){\line(0,1){2}}

\put(2,2){\line(1,-1){2}}
\put(2,2){\line(2,-1){4}}
\put(4,2){\line(1,-1){2}}

\put(0,0){\line(1,1){2}}
\put(0,4){\line(1,-1){2}}
\put(2,2){\line(1,1){2}}
\put(4,4){\line(1,-1){2}}

}

\put(18,0){

\put(0,4){\p} \put(2,4){\p} \put(4,4){\p} \put(6,4){\p}
\put(0,2){\p} \put(2,2){\p} \put(4,2){\p} \put(6,2){\p}
\put(0,0){\p}  \put(2,0){\p} \put(4,0){\p} \put(6,0){\p}

\put(0,0){\line(1,0){6}}
\put(0,2){\line(1,0){6}}
\put(0,4){\line(1,0){6}}
\put(0,0){\line(0,1){4}}
\put(2,0){\line(0,1){4}}
\put(6,0){\line(0,1){4}}
\put(4,2){\line(0,1){2}}

\put(2,2){\line(1,-1){2}}
\put(2,2){\line(2,-1){4}}
\put(4,2){\line(1,-1){2}}

\put(0,0){\line(1,1){2}}
\put(0,2){\line(1,1){2}}
\put(2,4){\line(1,-1){2}}
\put(4,2){\line(1,1){2}}

}

\put(27,0){

\put(0,4){\p} \put(2,4){\p} \put(4,4){\p} \put(6,4){\p}
\put(0,2){\p} \put(2,2){\p} \put(4,2){\p} \put(6,2){\p}
\put(0,0){\p}  \put(2,0){\p} \put(4,0){\p} \put(6,0){\p}

\put(0,0){\line(1,0){6}}
\put(0,2){\line(1,0){6}}
\put(0,4){\line(1,0){6}}
\put(0,0){\line(0,1){4}}
\put(2,0){\line(0,1){4}}
\put(6,0){\line(0,1){4}}
\put(4,2){\line(0,1){2}}

\put(2,2){\line(1,-1){2}}
\put(2,2){\line(2,-1){4}}
\put(4,2){\line(1,-1){2}}

\put(0,2){\line(1,-1){2}}
\put(0,4){\line(1,-1){2}}
\put(2,4){\line(1,-1){2}}
\put(4,2){\line(1,1){2}}

}

\end{picture}
\caption{Possible extensions in the situation $S_5$ in Figure~\ref{conflict-situations}.} \label{S5-possible}
\end{center}
\end{figure}

\begin{figure}[h]
\begin{center}
\begin{picture}(35,4.5)

\put(0,0){

\put(0,4){\p} \put(2,4){\p} \put(4,4){\p} \put(6,4){\p}
\put(0,2){\p} \put(2,2){\p} \put(4,2){\p} \put(6,2){\p}
\put(0,0){\p}  \put(2,0){\p} \put(4,0){\p} \put(6,0){\p}

\put(-0.6,4.2){?}  \put(1.9,4.2){2} \put(3.9,4.2){3} \put(6.2,4.2){2}
\put(-0.6,2.1){*} \put(1.5,2.2){1} \put(4.1,2.2){2} \put(6.2,2){1}
\put(-0.6,0.2){3}  \put(1.4,0.2){2}    \put(3.7,0.2){3} \put(6.2,0.2){4} 

\put(0,0){\line(1,0){6}}
\put(0,2){\line(1,0){6}}
\put(0,4){\line(1,0){6}}
\put(6,0){\line(0,1){4}}
\put(4,2){\line(0,1){2}}
\put(2,0){\line(0,1){4}}
\put(0,0){\line(0,1){4}}

\put(2,2){\line(1,-1){2}}
\put(2,2){\line(2,-1){4}}
\put(4,2){\line(1,-1){2}}

\put(0,0){\line(1,1){4}}
\put(0,2){\line(1,1){2}}
\put(4,4){\line(1,-1){2}}

}

\put(9,0){

\put(0,4){\p} \put(2,4){\p} \put(4,4){\p} \put(6,4){\p}
\put(0,2){\p} \put(2,2){\p} \put(4,2){\p} \put(6,2){\p}
\put(0,0){\p}  \put(2,0){\p} \put(4,0){\p} \put(6,0){\p}

\put(-0.6,4.2){3}  \put(1.9,4.2){2} \put(3.9,4.2){3} \put(6.2,4.2){2}
\put(-0.6,2.1){*} \put(1.5,2.2){1} \put(4.1,2.2){2} \put(6.2,2){1}
\put(-0.6,0.2){?}  \put(1.4,0.2){2}    \put(3.7,0.2){3} \put(6.2,0.2){4}

\put(0,0){\line(1,0){6}}
\put(0,2){\line(1,0){6}}
\put(0,4){\line(1,0){6}}
\put(6,0){\line(0,1){4}}
\put(4,2){\line(0,1){2}}
\put(2,0){\line(0,1){4}}
\put(0,0){\line(0,1){4}}

\put(2,2){\line(1,-1){2}}
\put(2,2){\line(2,-1){4}}
\put(4,2){\line(1,-1){2}}

\put(4,4){\line(1,-1){2}}
\put(2,2){\line(1,1){2}}
\put(0,2){\line(1,-1){2}}
\put(0,4){\line(1,-1){2}}

}

\put(18,0){

\put(0,4){\p} \put(2,4){\p} \put(4,4){\p} \put(6,4){\p}
\put(0,2){\p} \put(2,2){\p} \put(4,2){\p} \put(6,2){\p}
\put(0,0){\p}  \put(2,0){\p} \put(4,0){\p} \put(6,0){\p}

\put(-0.6,4.2){2}  \put(1.9,4.2){3} \put(3.9,4.2){1} \put(6.2,4.2){3}
\put(-0.6,2.2){3} \put(1.5,2.2){1} \put(4.1,2.2){2} \put(6.2,2){1}
\put(-0.6,0.1){*}  \put(1.4,0.2){2}    \put(3.7,0.2){3} \put(6.2,0.2){4} 

\put(0,0){\line(1,0){6}}
\put(0,2){\line(1,0){6}}
\put(0,4){\line(1,0){6}}
\put(6,0){\line(0,1){4}}
\put(4,2){\line(0,1){2}}
\put(2,0){\line(0,1){4}}
\put(0,0){\line(0,1){4}}

\put(2,2){\line(1,-1){2}}
\put(2,2){\line(2,-1){4}}
\put(4,2){\line(1,-1){2}}

\put(2,4){\line(1,-1){2}}
\put(4,2){\line(1,1){2}}
\put(0,0){\line(1,1){2}}
\put(0,4){\line(1,-1){2}}

}

\put(27,0){

\put(0,4){\p} \put(2,4){\p} \put(4,4){\p} \put(6,4){\p}
\put(0,2){\p} \put(2,2){\p} \put(4,2){\p} \put(6,2){\p}
\put(0,0){\p}  \put(2,0){\p} \put(4,0){\p} \put(6,0){\p}

\put(-0.6,4.2){?}  \put(1.9,4.2){3} \put(3.9,4.2){1} \put(6.2,4.2){3}
\put(-0.6,2.1){*} \put(1.5,2.2){1} \put(4.1,2.2){2} \put(6.2,2){1}
\put(-0.6,0.2){?}  \put(1.4,0.2){2}    \put(3.7,0.2){3} \put(6.2,0.2){4} 

\put(0,0){\line(1,0){6}}
\put(0,2){\line(1,0){6}}
\put(0,4){\line(1,0){6}}
\put(6,0){\line(0,1){4}}
\put(4,2){\line(0,1){2}}
\put(2,0){\line(0,1){4}}
\put(0,0){\line(0,1){4}}

\put(2,2){\line(1,-1){2}}
\put(2,2){\line(2,-1){4}}
\put(4,2){\line(1,-1){2}}

\put(2,4){\line(1,-1){2}}
\put(4,2){\line(1,1){2}}
\put(0,2){\line(1,-1){2}}
\put(0,2){\line(1,1){2}}

}

\end{picture}
\caption{Impossible extensions in the situation $S_5$ in Figure~\ref{conflict-situations}.} \label{S5-impossible}
\end{center}
\end{figure}
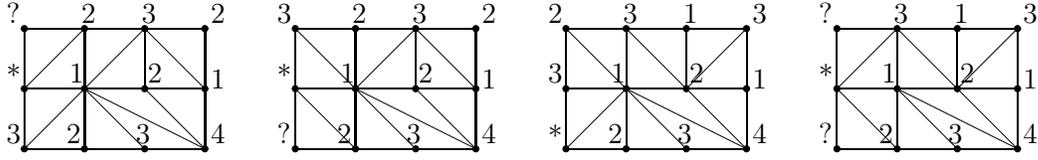

\noindent 
{\bf Situation $S_6$.} This situation is the same as situation $S_4$, since in both cases we have the same graph with three different colours in the top row. \\

\noindent 
{\bf Situation $S_7$.} Note that the subcases in Figure~\ref{S7-impossible} are not possible in this case, because the vertices labeled by * require usage of colour 4 contradicting the choice of the vertex $v$. Thus, in this situation, we only have two subcases obtained from $A_6$ and $A_7$ in Figure~\ref{non-repr-minim-triang}, respectively, by flipping with respect to a horizontal line, and rotating 180 degrees.\\

\begin{figure}[h]
\begin{center}
\begin{picture}(20,4.5)

\put(0,0){

                    \put(2,4){\p} \put(4,4){\p} \put(6,4){\p}
\put(0,2){\p} \put(2,2){\p} \put(4,2){\p} \put(6,2){\p}
\put(0,0){\p} \put(2,0){\p} \put(4,0){\p} \put(6,0){\p}

                \put(1.5,4.2){?} \put(3.9,4.1){*} \put(6.2,4.2){?}
\put(-0.6,2.2){2} \put(1.4,2.3){1} \put(4.1,2.2){2} \put(6.2,2){3}
 \put(-0.6,0.2){3}    \put(1.4,0.2){1} \put(3.4,0.2){4} \put(6.2,0.2){?} 

\put(0,0){\line(1,0){6}}
\put(0,2){\line(1,0){6}}
\put(2,4){\line(1,0){4}}
\put(6,0){\line(0,1){4}}
\put(4,0){\line(0,1){4}}
\put(2,2){\line(0,1){2}}
\put(0,0){\line(0,1){2}}

\put(0,0){\line(1,1){2}}
\put(0,0){\line(2,1){4}}
\put(2,0){\line(1,1){2}}
\put(4,0){\line(1,1){2}}

\put(2,2){\line(1,1){2}}
\put(4,4){\line(1,-1){2}}

}

\put(9,0){

                    \put(2,4){\p} \put(4,4){\p} \put(6,4){\p}
\put(0,2){\p} \put(2,2){\p} \put(4,2){\p} \put(6,2){\p}
\put(0,0){\p} \put(2,0){\p} \put(4,0){\p} \put(6,0){\p}

                \put(1.5,4.2){3} \put(3.9,4.2){1} \put(6.2,4.1){*}
\put(-0.6,2.2){2} \put(1.4,2.3){1} \put(4.1,2.2){2} \put(6.2,2){3}
 \put(-0.6,0.2){3}    \put(1.4,0.2){1} \put(3.4,0.2){4} \put(6.2,0.2){?} 

\put(0,0){\line(1,0){6}}
\put(0,2){\line(1,0){6}}
\put(2,4){\line(1,0){4}}
\put(6,0){\line(0,1){4}}
\put(4,0){\line(0,1){4}}
\put(2,2){\line(0,1){2}}
\put(0,0){\line(0,1){2}}

\put(0,0){\line(1,1){2}}
\put(0,0){\line(2,1){4}}
\put(2,0){\line(1,1){2}}
\put(4,0){\line(1,1){2}}
\put(2,4){\line(1,-1){2}}
\put(4,2){\line(1,1){2}}

}

\end{picture}
\caption{Impossible subcases in the situation $S_7$ in Figure~\ref{conflict-situations}.} \label{S7-impossible}
\end{center}
\end{figure}

\noindent 
{\bf Situation $S_8$.} We either have a copy of $B_1$ or $B_2$ flipped with respect to a vertical line, or we have one of the two subcases presented in Figure~\ref{S8-subcases}. Note that we can assume  in Figure~\ref{S8-subcases} that the vertices coloured by 1 (without loss of generality) are indeed of the same colour, since otherwise we would have the situation similar to that in Figure~\ref{S4-impossible}, which is impossible. However, that means that the vertex $v$ coloured in 4 is not in the leftmost column, and we can consider eight subcases of extending the graphs in Figure~\ref{S8-subcases} to the left: four extensions of the leftmost (resp., rightmost) graph are presented in Figure~\ref{S8-ext-1}  (resp., Figure~\ref{S8-ext-2}).

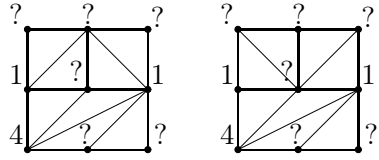
\begin{figure}[h]
\begin{center}
\begin{picture}(13,4)

\put(0,0){

\put(0,4){\p} \put(2,4){\p} \put(4,4){\p} 
\put(0,2){\p} \put(2,2){\p} \put(4,2){\p} 
\put(0,0){\p} \put(2,0){\p} \put(4,0){\p}

\put(-0.6,4.2){?}  \put(1.8,4.2){?} \put(4.1,4.1){?} 
\put(-0.6,2.2){1}  \put(1.4,2.3){?} \put(4.1,2.2){1} 
 \put(-0.6,0.2){4}  \put(1.7,0.2){?} \put(4.2,0.2){?} 

\put(0,0){\line(1,0){4}}
\put(0,2){\line(1,0){4}}
\put(0,4){\line(1,0){4}}
\put(4,0){\line(0,1){4}}
\put(2,2){\line(0,1){2}}
\put(0,0){\line(0,1){4}}

\put(2,4){\line(1,-1){2}}
\put(0,2){\line(1,1){2}}
\put(0,0){\line(1,1){2}}
\put(0,0){\line(2,1){4}}
\put(2,0){\line(1,1){2}}

}

\put(7,0){

\put(0,4){\p} \put(2,4){\p} \put(4,4){\p} 
\put(0,2){\p} \put(2,2){\p} \put(4,2){\p} 
\put(0,0){\p} \put(2,0){\p} \put(4,0){\p}

\put(-0.6,4.2){?}  \put(1.8,4.2){?} \put(4.1,4.1){?} 
\put(-0.6,2.2){1}  \put(1.4,2.3){?} \put(4.1,2.2){1} 
 \put(-0.6,0.2){4}  \put(1.7,0.2){?} \put(4.2,0.2){?} 

\put(0,0){\line(1,0){4}}
\put(0,2){\line(1,0){4}}
\put(0,4){\line(1,0){4}}
\put(4,0){\line(0,1){4}}
\put(2,2){\line(0,1){2}}
\put(0,0){\line(0,1){4}}

\put(2,2){\line(1,1){2}}
\put(2,2){\line(-1,1){2}}
\put(0,0){\line(1,1){2}}
\put(0,0){\line(2,1){4}}
\put(2,0){\line(1,1){2}}

}

\end{picture}
\caption{Subcases of interest in the situation $S_8$ in Figure~\ref{conflict-situations}.} \label{S8-subcases}
\end{center}
\end{figure}

Regarding the four graphs in Figure~\ref{S8-ext-1} considered from left to right one by one:
\begin{itemize}
\item Contains a copy of $A_8$ flipped with respect to a horizontal line.
\item Contains a copy of $A_3$ flipped with respect to a vertical line.
\item Contains a vertex marked by * that requires colour 4; thus this situation is impossible because of our choice of the vertex $v$.
\item If the bottom leftmost vertex coloured by 2 would be in the leftmost column, we could colour it by 1, and usage of colour 4 for colouring $v$ would be unnecessary. Thus, the graph can be extended to the left, and out of four possible extensions, two contain (rotated) copies of $T_1$ or $T_2$, while the other two presented in Figure~\ref{S8-ext-ext-1} are simply impossible, because they contain a vertex, marked by *, requiring colour 4. 
\end{itemize} 

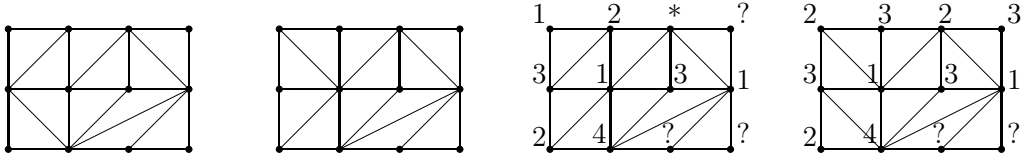
\begin{figure}[h]
\begin{center}
\begin{picture}(35,4.5)

\put(0,0){

\put(0,4){\p} \put(2,4){\p} \put(4,4){\p} \put(6,4){\p}
\put(0,2){\p} \put(2,2){\p} \put(4,2){\p} \put(6,2){\p}
\put(0,0){\p}  \put(2,0){\p} \put(4,0){\p} \put(6,0){\p}

\put(0,0){\line(1,0){6}}
\put(0,2){\line(1,0){6}}
\put(0,4){\line(1,0){6}}
\put(6,0){\line(0,1){4}}
\put(4,2){\line(0,1){2}}
\put(2,0){\line(0,1){4}}
\put(0,0){\line(0,1){4}}

\put(2,0){\line(1,1){2}}
\put(2,0){\line(2,1){4}}
\put(4,0){\line(1,1){2}}

\put(0,2){\line(1,-1){2}}
\put(0,2){\line(1,1){2}}
\put(2,2){\line(1,1){2}}
\put(4,4){\line(1,-1){2}}

}

\put(9,0){

\put(0,4){\p} \put(2,4){\p} \put(4,4){\p} \put(6,4){\p}
\put(0,2){\p} \put(2,2){\p} \put(4,2){\p} \put(6,2){\p}
\put(0,0){\p}  \put(2,0){\p} \put(4,0){\p} \put(6,0){\p}

\put(0,0){\line(1,0){6}}
\put(0,2){\line(1,0){6}}
\put(0,4){\line(1,0){6}}
\put(6,0){\line(0,1){4}}
\put(4,2){\line(0,1){2}}
\put(2,0){\line(0,1){4}}
\put(0,0){\line(0,1){4}}

\put(2,0){\line(1,1){2}}
\put(2,0){\line(2,1){4}}
\put(4,0){\line(1,1){2}}

\put(0,0){\line(1,1){2}}
\put(0,4){\line(1,-1){2}}
\put(2,2){\line(1,1){2}}
\put(4,4){\line(1,-1){2}}

}

\put(18,0){

\put(0,4){\p} \put(2,4){\p} \put(4,4){\p} \put(6,4){\p}
\put(0,2){\p} \put(2,2){\p} \put(4,2){\p} \put(6,2){\p}
\put(0,0){\p}  \put(2,0){\p} \put(4,0){\p} \put(6,0){\p}

\put(-0.6,4.2){1}  \put(1.9,4.2){2} \put(3.9,4.1){*} \put(6.2,4.2){?}
\put(-0.6,2.2){3} \put(1.5,2.2){1} \put(4.1,2.2){3} \put(6.2,2){1}
\put(-0.6,0.1){2}  \put(1.4,0.2){4}    \put(3.7,0.2){?} \put(6.2,0.2){?} 

\put(0,0){\line(1,0){6}}
\put(0,2){\line(1,0){6}}
\put(0,4){\line(1,0){6}}
\put(6,0){\line(0,1){4}}
\put(4,2){\line(0,1){2}}
\put(2,0){\line(0,1){4}}
\put(0,0){\line(0,1){4}}

\put(2,0){\line(1,1){2}}
\put(2,0){\line(2,1){4}}
\put(4,0){\line(1,1){2}}

\put(0,0){\line(1,1){2}}
\put(0,2){\line(1,1){2}}
\put(2,2){\line(1,1){2}}
\put(4,4){\line(1,-1){2}}

}

\put(27,0){

\put(0,4){\p} \put(2,4){\p} \put(4,4){\p} \put(6,4){\p}
\put(0,2){\p} \put(2,2){\p} \put(4,2){\p} \put(6,2){\p}
\put(0,0){\p}  \put(2,0){\p} \put(4,0){\p} \put(6,0){\p}

\put(-0.6,4.2){2}  \put(1.9,4.2){3} \put(3.9,4.2){2} \put(6.2,4.2){3}
\put(-0.6,2.2){3} \put(1.5,2.2){1} \put(4.1,2.2){3} \put(6.2,2){1}
\put(-0.6,0.1){2}  \put(1.4,0.2){4}    \put(3.7,0.2){?} \put(6.2,0.2){?} 

\put(0,0){\line(1,0){6}}
\put(0,2){\line(1,0){6}}
\put(0,4){\line(1,0){6}}
\put(6,0){\line(0,1){4}}
\put(4,2){\line(0,1){2}}
\put(2,0){\line(0,1){4}}
\put(0,0){\line(0,1){4}}

\put(2,0){\line(1,1){2}}
\put(2,0){\line(2,1){4}}
\put(4,0){\line(1,1){2}}

\put(0,2){\line(1,-1){2}}
\put(0,4){\line(1,-1){2}}
\put(2,2){\line(1,1){2}}
\put(4,4){\line(1,-1){2}}

}

\end{picture}
\caption{All possible extensions to the left of the leftmost graph in Figure~\ref{S8-subcases}.} \label{S8-ext-1}
\end{center}
\end{figure}

Regarding the four graphs in Figure~\ref{S8-ext-2} considered from left to right one by one:
\begin{itemize}
\item Contains a copy of $A_4$ flipped with respect to a vertical line.
\item Contains a copy of $A_5$ rotated 180 degrees.
\item Contains a vertex marked by * that requires colour 4; thus this situation is impossible because of our choice of the vertex $v$.
\item If the bottom leftmost vertex coloured by 2 would be in the leftmost column, we could colour it by 1, and usage of colour 4 for colouring $v$ would be unnecessary. Thus, the graph can be extended to the left, and out of four possible extensions, two contain (rotated) copies of $T_2$, while the other two presented in Figure~\ref{S8-ext-ext-2} are simply impossible, because they contain a vertex, marked by *, requiring colour 4. 
\end{itemize} 

\begin{figure}[h]
\begin{center}
\begin{picture}(35,4.5)

\put(0,0){

\put(0,4){\p} \put(2,4){\p} \put(4,4){\p} \put(6,4){\p}
\put(0,2){\p} \put(2,2){\p} \put(4,2){\p} \put(6,2){\p}
\put(0,0){\p}  \put(2,0){\p} \put(4,0){\p} \put(6,0){\p}

\put(0,0){\line(1,0){6}}
\put(0,2){\line(1,0){6}}
\put(0,4){\line(1,0){6}}
\put(6,0){\line(0,1){4}}
\put(4,2){\line(0,1){2}}
\put(2,0){\line(0,1){4}}
\put(0,0){\line(0,1){4}}

\put(2,0){\line(1,1){2}}
\put(2,0){\line(2,1){4}}
\put(4,0){\line(1,1){2}}

\put(0,0){\line(1,1){2}}
\put(0,2){\line(1,1){2}}
\put(2,4){\line(1,-1){2}}
\put(4,2){\line(1,1){2}}

}

\put(9,0){

\put(0,4){\p} \put(2,4){\p} \put(4,4){\p} \put(6,4){\p}
\put(0,2){\p} \put(2,2){\p} \put(4,2){\p} \put(6,2){\p}
\put(0,0){\p}  \put(2,0){\p} \put(4,0){\p} \put(6,0){\p}

\put(0,0){\line(1,0){6}}
\put(0,2){\line(1,0){6}}
\put(0,4){\line(1,0){6}}
\put(6,0){\line(0,1){4}}
\put(4,2){\line(0,1){2}}
\put(2,0){\line(0,1){4}}
\put(0,0){\line(0,1){4}}

\put(2,0){\line(1,1){2}}
\put(2,0){\line(2,1){4}}
\put(4,0){\line(1,1){2}}

\put(0,2){\line(1,-1){2}}
\put(0,4){\line(1,-1){2}}
\put(2,4){\line(1,-1){2}}
\put(4,2){\line(1,1){2}}

}

\put(18,0){

\put(0,4){\p} \put(2,4){\p} \put(4,4){\p} \put(6,4){\p}
\put(0,2){\p} \put(2,2){\p} \put(4,2){\p} \put(6,2){\p}
\put(0,0){\p}  \put(2,0){\p} \put(4,0){\p} \put(6,0){\p}

\put(-0.6,4.2){2}  \put(1.9,4.1){*} \put(3.9,4.2){1} \put(6.2,4.2){2}
\put(-0.6,2.2){3} \put(1.5,2.2){1} \put(4.1,2.2){3} \put(6.2,2){1}
\put(-0.6,0.1){2}  \put(1.4,0.2){4}    \put(3.7,0.2){?} \put(6.2,0.2){?} 

\put(0,0){\line(1,0){6}}
\put(0,2){\line(1,0){6}}
\put(0,4){\line(1,0){6}}
\put(6,0){\line(0,1){4}}
\put(4,2){\line(0,1){2}}
\put(2,0){\line(0,1){4}}
\put(0,0){\line(0,1){4}}

\put(2,0){\line(1,1){2}}
\put(2,0){\line(2,1){4}}
\put(4,0){\line(1,1){2}}

\put(0,0){\line(1,1){2}}
\put(0,4){\line(1,-1){2}}
\put(2,4){\line(1,-1){2}}
\put(4,2){\line(1,1){2}}

}

\put(27,0){

\put(0,4){\p} \put(2,4){\p} \put(4,4){\p} \put(6,4){\p}
\put(0,2){\p} \put(2,2){\p} \put(4,2){\p} \put(6,2){\p}
\put(0,0){\p}  \put(2,0){\p} \put(4,0){\p} \put(6,0){\p}

\put(-0.6,4.2){1}  \put(1.9,4.2){2} \put(3.9,4.2){1} \put(6.2,4.2){2}
\put(-0.6,2.2){3} \put(1.5,2.2){1} \put(4.1,2.2){3} \put(6.2,2){1}
\put(-0.6,0.1){2}  \put(1.4,0.2){4}    \put(3.7,0.2){?} \put(6.2,0.2){?} 

\put(0,0){\line(1,0){6}}
\put(0,2){\line(1,0){6}}
\put(0,4){\line(1,0){6}}
\put(6,0){\line(0,1){4}}
\put(4,2){\line(0,1){2}}
\put(2,0){\line(0,1){4}}
\put(0,0){\line(0,1){4}}

\put(2,0){\line(1,1){2}}
\put(2,0){\line(2,1){4}}
\put(4,0){\line(1,1){2}}

\put(0,2){\line(1,-1){2}}
\put(0,2){\line(1,1){2}}
\put(2,4){\line(1,-1){2}}
\put(4,2){\line(1,1){2}}

}

\end{picture}
\caption{All possible extensions to the left of the rightmost graph in Figure~\ref{S8-subcases}.} \label{S8-ext-2}
\end{center}
\end{figure}
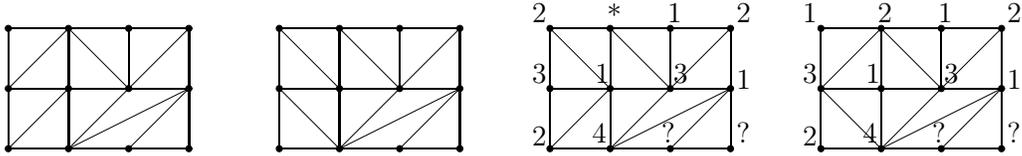

\begin{figure}[h]
\begin{center}
\begin{picture}(20,4.5)

\put(0,0){

\put(0,4){\p} \put(2,4){\p} \put(4,4){\p} \put(6,4){\p} \put(8,4){\p}
\put(0,2){\p} \put(2,2){\p} \put(4,2){\p} \put(6,2){\p} \put(8,2){\p}
\put(0,0){\p}  \put(2,0){\p} \put(4,0){\p} \put(6,0){\p} \put(8,0){\p}

\put(-0.6,4.2){3}  \put(1.4,4.2){2}  \put(3.9,4.2){3} \put(5.9,4.2){2} \put(8.2,4.2){3}
\put(-0.6,2.2){1} \put(1.4,2.2){3} \put(3.5,2.2){1} \put(6.1,2.2){3} \put(8.2,2){1}
\put(-0.6,0.1){*}  \put(1.4,0.2){2}  \put(3.4,0.2){4}    \put(5.7,0.2){?} \put(8.2,0.2){?} 

\put(0,0){\line(1,0){8}}
\put(0,2){\line(1,0){8}}
\put(0,4){\line(1,0){8}}
\put(8,0){\line(0,1){4}}
\put(6,2){\line(0,1){2}}
\put(4,0){\line(0,1){4}}
\put(2,0){\line(0,1){4}}
\put(0,0){\line(0,1){4}}

\put(4,0){\line(1,1){2}}
\put(4,0){\line(2,1){4}}
\put(6,0){\line(1,1){2}}

\put(2,2){\line(1,-1){2}}
\put(2,4){\line(1,-1){2}}
\put(4,2){\line(1,1){2}}
\put(6,4){\line(1,-1){2}}
\put(0,0){\line(1,1){2}}
\put(0,2){\line(1,1){2}}

}

\put(11,0){

\put(0,4){\p} \put(2,4){\p} \put(4,4){\p} \put(6,4){\p} \put(8,4){\p}
\put(0,2){\p} \put(2,2){\p} \put(4,2){\p} \put(6,2){\p} \put(8,2){\p}
\put(0,0){\p}  \put(2,0){\p} \put(4,0){\p} \put(6,0){\p} \put(8,0){\p}

\put(-0.6,4.1){*}  \put(1.4,4.2){2}  \put(3.9,4.2){3} \put(5.9,4.2){2} \put(8.2,4.2){3}
\put(-0.6,2.2){1} \put(2.1,2.2){3} \put(3.5,2.2){1} \put(6.1,2.2){3} \put(8.2,2){1}
\put(-0.6,0.2){3}  \put(2.1,0.2){2}  \put(3.4,0.2){4}    \put(5.7,0.2){?} \put(8.2,0.2){?} 

\put(0,0){\line(1,0){8}}
\put(0,2){\line(1,0){8}}
\put(0,4){\line(1,0){8}}
\put(8,0){\line(0,1){4}}
\put(6,2){\line(0,1){2}}
\put(4,0){\line(0,1){4}}
\put(2,0){\line(0,1){4}}
\put(0,0){\line(0,1){4}}

\put(4,0){\line(1,1){2}}
\put(4,0){\line(2,1){4}}
\put(6,0){\line(1,1){2}}

\put(2,2){\line(1,-1){2}}
\put(2,4){\line(1,-1){2}}
\put(4,2){\line(1,1){2}}
\put(6,4){\line(1,-1){2}}
\put(0,2){\line(1,-1){2}}
\put(0,4){\line(1,-1){2}}

}

\end{picture}
\caption{Possible $T_1$, $T_2$-avoiding extensions to the left of the rightmost graph in Figure~\ref{S8-ext-1}.} \label{S8-ext-ext-1}
\end{center}
\end{figure}
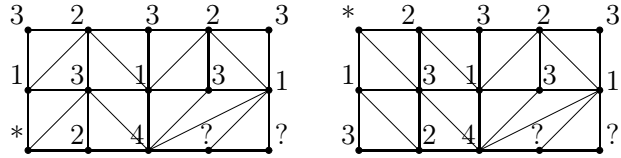

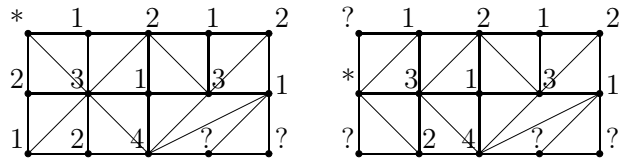
\begin{figure}[h]
\begin{center}
\begin{picture}(20,4.5)

\put(0,0){

\put(0,4){\p} \put(2,4){\p} \put(4,4){\p} \put(6,4){\p} \put(8,4){\p}
\put(0,2){\p} \put(2,2){\p} \put(4,2){\p} \put(6,2){\p} \put(8,2){\p}
\put(0,0){\p}  \put(2,0){\p} \put(4,0){\p} \put(6,0){\p} \put(8,0){\p}

\put(-0.6,4.1){*}  \put(1.4,4.2){1}  \put(3.9,4.2){2} \put(5.9,4.2){1} \put(8.2,4.2){2}
\put(-0.6,2.2){2} \put(1.4,2.2){3} \put(3.5,2.2){1} \put(6.1,2.2){3} \put(8.2,2){1}
\put(-0.6,0.2){1}  \put(1.4,0.2){2}  \put(3.4,0.2){4}    \put(5.7,0.2){?} \put(8.2,0.2){?} 

\put(0,0){\line(1,0){8}}
\put(0,2){\line(1,0){8}}
\put(0,4){\line(1,0){8}}
\put(8,0){\line(0,1){4}}
\put(6,2){\line(0,1){2}}
\put(4,0){\line(0,1){4}}
\put(2,0){\line(0,1){4}}
\put(0,0){\line(0,1){4}}

\put(4,0){\line(1,1){2}}
\put(4,0){\line(2,1){4}}
\put(6,0){\line(1,1){2}}

\put(2,2){\line(1,-1){2}}
\put(2,2){\line(1,1){2}}
\put(4,4){\line(1,-1){2}}
\put(6,2){\line(1,1){2}}
\put(0,0){\line(1,1){2}}
\put(0,4){\line(1,-1){2}}

}

\put(11,0){

\put(0,4){\p} \put(2,4){\p} \put(4,4){\p} \put(6,4){\p} \put(8,4){\p}
\put(0,2){\p} \put(2,2){\p} \put(4,2){\p} \put(6,2){\p} \put(8,2){\p}
\put(0,0){\p}  \put(2,0){\p} \put(4,0){\p} \put(6,0){\p} \put(8,0){\p}

\put(-0.6,4.2){?}  \put(1.4,4.2){1}  \put(3.9,4.2){2} \put(5.9,4.2){1} \put(8.2,4.2){2}
\put(-0.6,2.1){*} \put(1.5,2.2){3} \put(3.5,2.2){1} \put(6.1,2.2){3} \put(8.2,2){1}
\put(-0.6,0.2){?}  \put(2.1,0.2){2}  \put(3.4,0.2){4}    \put(5.7,0.2){?} \put(8.2,0.2){?} 

\put(0,0){\line(1,0){8}}
\put(0,2){\line(1,0){8}}
\put(0,4){\line(1,0){8}}
\put(8,0){\line(0,1){4}}
\put(6,2){\line(0,1){2}}
\put(4,0){\line(0,1){4}}
\put(2,0){\line(0,1){4}}
\put(0,0){\line(0,1){4}}

\put(4,0){\line(1,1){2}}
\put(4,0){\line(2,1){4}}
\put(6,0){\line(1,1){2}}

\put(2,2){\line(1,-1){2}}
\put(2,2){\line(1,1){2}}
\put(4,4){\line(1,-1){2}}
\put(6,2){\line(1,1){2}}
\put(0,2){\line(1,-1){2}}
\put(0,2){\line(1,1){2}}

}

\end{picture}
\caption{Possible $T_2$-avoiding extensions to the left of the rightmost graph in Figure~\ref{S8-ext-2}.} \label{S8-ext-ext-2}
\end{center}
\end{figure}

Thus, we proved, that if a triangulation of a rectangular polyomino with a single domino tile is not 3-colourable, then it must contain a graph from $S$ as an induced subgraph. \end{proof}

Theorem~\ref{main-thm} now follows from Lemma~\ref{lemma-grid} taking into account the fact that all graphs in $S$ are non-word-representable.

\section{Directions of further research}\label{final-remarks-sec}  

Natural directions of further research are as follows. 

\begin{itemize}
\item Does Theorem~\ref{main-thm} hold if we allow more than one domino tile?  Note that using current approach, the analysis involved seems to require too many cases to be considered. In either case, the problem has the following particular subproblem:
\begin{itemize}
\item Does Theorem~\ref{main-thm} hold if we allow just horizontal domino tiles (equivalently, just vertical domino tiles)?
\end{itemize}
\item Does Theorem~\ref{main-thm} hold if the domino tile is placed on other, not necessarily rectangular, convex polyominoes? What about allowing more than one domino tile to be used? Note that the same counterexample as in~\cite{Akrobotu} can be used to show that Theorem~\ref{main-thm} does not hold for non-convex polyominoes. 
\end{itemize}

\section*{Acknowledgements}

The authors are grateful to the anonymous referee for providing very useful comments that improved the presentation of the paper.

\end{document}